\documentclass[fleqn]{scrartcl}
\setkomafont{disposition}{\normalfont\bfseries}
\makeatletter\let\TTTemp\cap\makeatother
\usepackage[greek,english]{babel}
\usepackage{teubner}
\makeatletter\let\cap\TTTemp\makeatother
\usepackage{eurosym}
\usepackage{amssymb}
\usepackage{amsmath}
\usepackage{amsfonts}
\usepackage{lgreek}
\usepackage{graphicx}
\usepackage{textcomp}
\usepackage{cancel}
\usepackage{soul}
\usepackage{cjhebrew}
\usepackage{mathtools}
\usepackage{rotating}
\usepackage{esvect}
\usepackage{natbib}
\usepackage{MnSymbol}
\usepackage{centernot}
\usepackage{cancel}
\usepackage{scalerel}
\usepackage{hyperref}
\usepackage{pict2e}
\usepackage[nottoc,numbib]{tocbibind}
\usepackage{bbding}
\usepackage{mathtools}
\usepackage{amsthm}
\swapnumbers

\makeatletter
\DeclareRobustCommand{\cvdash}{%
  \mathrel{\mathpalette\cvd@sh\relax}
}

\newcommand{\cvd@sh}[2]{%
  \sbox\z@{$\m@th#1\vdash$}%
  \setlength{\unitlength}{1.1\wd\z@}%
  \begin{picture}(1,0.75)
  \roundcap\roundjoin
  \polyline(0.125,0)(0.4,0)(0.4,0.75)(0.125,0.75)
  \polyline(0.4,0.375)(0.925,0.375)
  \end{picture}%
}

\DeclareRobustCommand{\cVdash}{%
  \mathrel{\mathpalette\cVd@sh\relax}
}
\newcommand{\cVd@sh}[2]{%
  \sbox\z@{$\m@th#1\vdash$}%
  \setlength{\unitlength}{1.1\wd\z@}%
  \begin{picture}(1,0.75)
  \roundcap\roundjoin
  \polyline(0.125,0)(0.4,0)(0.4,0.75)(0.125,0.75)(0.125,0)
  \polyline(0.4,0.375)(0.925,0.375)
  \end{picture}%
}

\DeclareRobustCommand{\cVvdash}{%
  \mathrel{\mathpalette\cVvd@sh\relax}%
}

\newcommand{\cVvd@sh}[2]{%
  \sbox\z@{$\m@th#1\vdash$}%
  \setlength{\unitlength}{1.1\wd\z@}%
  \begin{picture}(1,0.75)
  \roundcap\roundjoin
  \polyline(0.125,0)(0.55,0)(0.55,0.75)(0.125,0.75)(0.125,0)
  \polyline(.3375,0)(.3375,0.75)
  \polyline(0.6,0.375)(0.925,0.375)
  \end{picture}%
}
\makeatother

\allowdisplaybreaks

\newbox\gnBoxA
\newdimen\gnCornerHgt
\setbox\gnBoxA=\hbox{$\ulcorner$}
\global\gnCornerHgt=\ht\gnBoxA
\newdimen\gnArgHgt
\def\Godelnum #1{%
\setbox\gnBoxA=\hbox{$#1$}%
\gnArgHgt=\ht\gnBoxA%
\ifnum     \gnArgHgt<\gnCornerHgt \gnArgHgt=0pt%
\else \advance \gnArgHgt by -\gnCornerHgt%
\fi \raise\gnArgHgt\hbox{$\ulcorner$} \box\gnBoxA %
\raise\gnArgHgt\hbox{$\urcorner$}}

\newcommand{\bleq}{\mathrel{\mathpalette\bleqinn\relax}}
\newcommand{\bleqinn}[2]{%
  \ooalign{%
    \raisebox{0.04ex}{\scalebox{1.24}[1.15]{$#1\blacktriangleleft$}}\cr
    $#1\leq$\cr
  }%
}

\newcommand{\blt}{\mathrel{\mathpalette\bltinn\relax}} \newcommand{\bltinn}[2]{ \ooalign{ \raisebox{0.04ex}{\scalebox{1.23}[1.22]{$#1\blacktriangleleft$}}\cr $#1$\cr } }

\newcommand\fat[1]{\ThisStyle{\hstretch{1.4}{\ooalign{%
  \kern.46pt$\SavedStyle#1$\cr\kern.33pt$\SavedStyle#1$\cr%
  \kern.2pt$\SavedStyle#1$\cr$\SavedStyle#1$}}}}
\def\fvee{\mathbin{\fat{\vee}}}

\newcommand*{\aoverb}[2]{\ensuremath{\genfrac{}{}{0pt}{}{#1}{#2}}}

\renewcommand*{\koppa}{\ensuremath{\text{\textgreek\coppa}}}
\renewcommand*{\Koppa}{\ensuremath{\text{\textgreek\Coppa}}}

\newtheorem{theorem}{Theorem}[section]

\newtheorem{cannon}[theorem]{Cannon}

\newtheorem{corollary}[theorem]{Corollary}

\newtheorem{definition}[theorem]{Definition}

\newtheorem{exercise}[theorem]{Exercise}

\newtheorem{lemma}[theorem]{Lemma}

\newtheorem{postulate}[theorem]{Postulate}

\newtheorem{proposition}[theorem]{Proposition}
\newtheorem{remark}[theorem]{Remark}

\title{Elements of Librationism}

\author{Frode Bj\o rdal}

\date{}

\begin{document}

\maketitle
\thispagestyle{empty}
\setcounter{page}{-2}
\thispagestyle{empty}

\newpage
\tableofcontents\thispagestyle{empty}

\newpage

\noindent

\section{Introduction}

$\cvdash\vdash$

The librationist system now named \pounds \ (libra) is detailed some in \citep{Frode2011} and \citep{Frode2012a},
but we go beyond those accounts here and inter alia focus more incisively upon aspects
related to how it gives rise to new perspectives concerning inferential
principles, and the nature of connectives. The purpose of \pounds \ is to deal with the semantical and set theoretical paradoxes in a novel manner, and it offers a novel set theoretic foundation of mathematical reasoning as well as a theory of truth in a semantically closed language withouth compromising classical logic. In the following presentation there are some complexities that need justification, and we will explain some of these to provide their motivation.

\pounds \ is primarily accounted for semantically. We work in a segment $\textbf{L}_{\varsigma}$ of G\"{o}del's constructible hierarchy $\textbf{L}$, and the  countable ordinal $\varsigma$ needed is the least $\Sigma_{3}$-admissible ordinal. One important aspect of our approach, as pointed out in Section \ref{Formations as numbers}, is that we by means of the austere formation rules of Section \ref{Fundamentals} are provided an external coding which identify formulas and expressions of \pounds \ with finite von Neumann ordinals of $\textbf{L}_{\varsigma}$; an important advantage of this approach is that we much more transparently and explicitly than usual find a covering ordinal for the Herzberger style semantic process using $\Sigma_{3}$-collection in the proof of theorem \ref{Herzberger} (i), and without any invocation of the L\"{o}wenheim-Skolem theorem or other notions which may be thought of as presupposing uncountable sets.

In Section \ref{Formations as numbers} a second coding layer is introduced to provide G\"{o}del codings of formulas of \pounds \ in \pounds \ itself, and this is done in such a way that for an expression $E$ of \pounds \ \textbf{L} believes a natural number is $E$ iff \pounds \ believes it is $\Godelnum{E}$. This involves some intricacies, but there are some major payoffs. The motivation for the following pedantic policies pertaining to the formal language of \pounds \ is stated in Remark \ref{codingremark} of Section \ref{Formations as numbers}, and these policies enable our statement of semantical principle $P5$ in Section \ref{The semantics of pounds} and the articulation of \textit{the truth prescription} in Section \ref{prescriptions}.

We use the term \textit{thesis} for a formula which is librationistically valid, and \textit{theorem} is used as expected for results concerning \pounds. We consider a formula $A$ an \textit{anti thesis} of a system iff its negjunction ("negation") $\lnot A$ is a thesis of the system. A theory $ T $ extends theory $ S $ \emph {soberly} iff the set of theses of $ S $ is a proper subset of the set of theses of $ T $ and no thesis of $ T $ is an anti thesis of $ S $. $ T $ soberly interprets $ S $ iff $ T $ interprets $ S $ and for no \emph{interpretans} $A^{i}$ in $ T $ of the \emph{interpretandum} $A$ of $ S $ does
$ T $ have $A^{i}$ as a thesis and an anti thesis. As it is, \pounds \ soberly extends classical logic and impredicative mathematics in the sense of reverse mathematics. 

The central matters we cover are
central to our understanding of thinking and rationality, and to a large extent confounding. We will point out that unlike other foundational systems that are on offer librationism is in a very precise sense a \textit{disconnectionist} point of view as \pounds \ has theses which are $disconnected$. \pounds \ is not a non-classical system as it extends classical logic \textit{soberly}. Unlike paraconsistent systems engendered by adding naive comprehension to paraconsistent logics suggested in the literature for dealing with the paradoxes, \pounds \ is neither inconsistent nor contradictory. As divulged in \cite{Frode2013} we show how \pounds \ is able to soberly interpret $ZF$. As divulged in \cite{Frode2014a} we show that \pounds \ also makes it possible to isolate the definable real numbers and so prepare a space for definable analysis. In section \ref{definable echelon} we show that \pounds \ has an interpretation of $ZFC$ given what we name the \textit{Skolem-Fraenkel Postulate} (Postulate 21.20).

Unlike in \citep{Frode2012a} and in earlier superseded accounts of the librationist approach we use \textit{prescription} where one would expect \textit{axiom schema} and \textit{regulation} where one would expect \textit{inference rule}; moreover, \textit{prescript} is used where one would expect \textit{(proper) axiom} and \textit{regula} should be used if one wants to denote a specific instance of a regulation. A \textit{prescribe} is a fundamental thesis which cannot be universally generalized. Just prescribes, prescripts and prescriptions are \textit{posits}. The terms \textit{cannon} and \textit{postulate} are for statements that one thinks should be provable as theorems though for reasons connected with G\"{o}del's second incompleteness theorem only by means much stronger than those used to isolate the notion of being a theorem of \pounds.

This and further terminology is introduced in order to make distinctions called for by new concepts, and also to underline the fact that \pounds \ is not an axiomatic or formal system in the common sense. \pounds \ is \textit{super-formal}, or \textit{semi-formal} according to what the author takes as unfortunate contemporary terminology. We suggest to also consider \pounds \ a \textit{contentual} system as it is categorical and so rather more unambiguous with respect to content than merely formal systems. Nevertheless, one may as much as $\omega$-\textit{logic} is partly denoted by the term ``logic'' also understand \pounds \ as a logic, especially as Frege in the opening phrases of \textit{Der Gedanke} nails down logic's subject matter to be truth itself. For truth is indeed a most central theme of \pounds, and it should be compared with axiomatic theories of truth as overviewed in \cite{sep-truth-axiomatic}; \pounds \ much extends such formal relatives, and provides its sentences with names.

Our considerations concerning librationism are inherently semantical, and soundness and completeness considerations are irrelevant. As we focus upon \textit{one} intended model considerations concerning \textit{compactness} are extraneous. The posits we isolate are only examples of formula schemas which hold librationistically and the regulations we display are likewise just examples of principles which tell us when thesishood is regulated from that of others.  Nevertheless, the prescriptions and prescripts and regulations we do isolate are quite informative and comprehensive, and they moreover provide librationist justifications for important axiomatic systems.

\section{Fundamentals of \pounds} \label{Fundamentals}

As in \citep{Frode2012a} we continue to consider \pounds \ a $\textit{theory  of  sorts}$ and thence also a $\textit{theory of properties}$. However, in the final analysis the language of \pounds \ may be taken as just that of the ordinary language of set theory without identity including set brackets plus two special extra sort constants for $truth$ and $enumeration$. We take sorts or sets that are designated by terms (\textit{pronomina}) that are formed (in a sense made more precise below) without using the alethizor or the enumerator to be \textit{sets}. Hence \pounds \ is also an alternative theory of sets, and it contains paradoxical sets as e.g. Russell's set $\{x|x\notin x\}$. So $pace$ G\"{o}del and others this author now thinks that there always were and always will be set theoretical paradoxes.

To avert confusion between object language and metalanguage statements we use the following boldface syntax for metalanguage statements to denote objects of $\textbf{L}_{\varsigma}$: parentheses \fat ( and \fat ), set brackets \fat \{ and \fat \}, element sign \fat \in, disjunction $\fvee$, conjunction \fat \&, negation $\fat \sim$, implication $\fat \Rightarrow$, biimplication $\fat \Leftrightarrow$, variables \fat x, \fat y, \fat z, ..., ordinals \fat \alpha, \fat \beta, \fat \gamma, ... existential quantifier \fat \Sigma \ and universal quantifier \fat \Pi; $\fat \prec$ and $\fat \preceq$ are the usual orderings of the ordinals of $\textbf{L}_{\varsigma}$ and $\mathbf{N}$ is its set of finite von Neumann ordinals or natural numbers.

The austere Polish language $\texttt{L}(\pounds)$ of \pounds \ is its alphabet
$\texttt{A}(\pounds)$ plus its formation rules $\texttt{F}(\pounds)$.
$\texttt{A}(\pounds)$ is the two signs  . (dot) and $ | $ (bar). $\texttt{F}(\pounds)$ is  $\texttt{F}^S(0) - \texttt{F}^S(3)$, $\texttt{F}^F(0) - \texttt{F}^F(9)$, $\texttt{F}^E(0) - \texttt{F}^E(9)$ plus $\texttt{F}^V(0) - \texttt{F}^V(4)$ as in the following. Minuscule letters from the beginning of the Latin alphabet range over terms and their capital counterparts range over formulas. We let $u$, $v$, $w$, $x$, $y$, $z$, $u'$ and so on stand for arbritrary \textit{noemata} (see the formation rules right below). We take $m$ and $n$ to range over natural numbers. For any numeral $ \underline{n} $ we let $ |_{\underline{n}} $ be bar followed by $ n $ dots so that $ |_{\underline{0}} $ is $ | $ and $ |_{\underline{n+1}} $ is $ |_{\underline{n}}.$; when just the latter convention is used in presenting we say that the presentation is in \textit{bare form}, and when the convention is suppressed so that only bars and dots are used we say that it is in \textit{austere form}.

\bigskip 
\noindent
$\texttt{F}^S(0)$: $ |_{\underline{0}} $ is the \textit{sortifier}.\\
$\texttt{F}^S(1)$: $ |_{\underline{1}} $ is the \textit{universalizor}.\\
$\texttt{F}^S(2)$: $ |_{\underline{2}} $ is the \textit{norifyer}.\\
$\texttt{F}^S(3)$: Just the sortifier, the universalizor and the norifyer are \textit{syncategoremata}.
\smallskip\\
$\texttt{F}^F(0)$: $ |_{\underline{3}} $ is the \textit{alethizor}.\\
$\texttt{F}^F(1)$: $ |_{\underline{4}} $ is the \textit{enumerator}.\\
$\texttt{F}^F(2)$: $ |_{\underline{5}} $ is a \textit{noema}. \\
$\texttt{F}^F(3)$: If $ |_{\underline{n}}$ is a noema then $ |_{\underline{n+1}}$ is a noema. \\
$\texttt{F}^F(4)$: Nothing else is a noema.\\
$\texttt{F}^F(5)$: Just the alethizor, the enumerator and noemata are \textit{praenomina}.\\
$\texttt{F}^F(6)$: Just syncategoremata and praenomina are \textit{symbols}.\\
$\texttt{F}^F(7)$: If $ p $  is a symbol then $ p $ is a \textit{formation}.\\
$\texttt{F}^F(8)$: If $ p $ and $ q $ are formations then $ pq $ is a formation.\\
$\texttt{F}^F(9)$: Nothing else is a formation.
\smallskip\\
$\texttt{F}^E(0)$: Praenomina are \textit{terms}. \\
$\texttt{F}^E(1)$: If $a$ and $b$ are terms then $ba$ is a \textit{formula}. \\
$\texttt{F}^E(2)$: If $A$ and $B$ are formulas then $ |_{\underline {2}}AB$ is a formula. \\
$\texttt{F}^E(3)$: If $a$ and $b$ are terms then $|_{\underline{2}}ba$ is a \textit{term}. \\
$\texttt{F}^E(4)$: If $A$ is a formula and $ y $ is a noema then $ |_{\underline {1}}yA$ is a formula. \\
$\texttt{F}^E(5)$: If $A$ is a formula and $ y $ is a noema then
  $ |_{\underline {0}}yA$ is a term. \\
$\texttt{F}^E(6)$: Nothing else is a term or a formula.\\
$\texttt{F}^E(7)$: Just terms and formulas are \textit{expressions}.\\
$\texttt{F}^E(8)$: Just formulas are \textit{sentences}.\\
$\texttt{F}^E(9)$: Just terms are \textit{nomina} and \textit{sort constants}.
\smallskip\\
$\texttt{F}^V(0)$: If $A$ is a formula and $ y $ is a noema then all occurrences of $y$ in $ |_{\underline {1}}yA$ are variables.\\
$\texttt{F}^V(1)$: If $A$ is a formula and $ y $ is a noema then all occurrences of $y$ in $ |_{\underline {0}}yA$ are variables.\\
$\texttt{F}^V(2)$: The first occurrence of $y$ in $ |_{\underline {1}}yA$, or $ |_{\underline {0}}yA$, is the binding variable.\\
$\texttt{F}^V(3)$: All occurrences of $y$ in $ |_{\underline {1}}yA$ and in $ |_{\underline {0}}yA$ are bound variables.\\
$\texttt{F}^V(4)$: Nothing else is a variable. 

\medskip

The semantical principles of \pounds \ in the final analysis compel us to the unusual policies embodied by $\texttt{F}^E(8)$ and $\texttt{F}^E(9)$ (cfr. section \ref{curries} and \ref{sec:nominism} for the deeper motivation). The resulting \textit{nominist} view supported by \pounds \ is accounted for in more detail in section \ref{sec:nominism}, and the related and necessitated \textit{nominist turn} is explained in section \ref{Negjunction}. \pounds \ as a consequence has no free \textit{variables}. We have instead opted for using the expressions \textit{noema }%
(singular) and \textit{noemata} (plural) where one would expect \textit{free variable}. This is inter alia justified by
the fact that one meaning of the word \textit{noema} as listed in the Oxford
English Dictionary is: \textit{A figure of speech whereby something stated
obscurely is nevertheless intended to be understood or worked out}. Also,
the Greek letter \begin{greek} n \end{greek} in the original Greek word \begin{greek} n'ohma \end{greek} typographically very much resembles lower case $ v$. We also use variants of Latin \textit{nomen} for terms.\footnote{It may be that our knowledge of proto Indo European does not entirely justify us in thinking that there is an etymological relatedness between Greek \begin{greek} n'ohma \end{greek} (\textit{noema}) and Greek \begin{greek} >'onoma \end{greek} (\textit{onoma}), but the terms are surely conceptually related enough e.g. in as far as we often come to know things by their names. Quite possibly, ancient Greek philosophers did not make a connection between these notions. However these matters may be, in \pounds \ we take noemata to also be nomina.}


We follow the convention that the expression ``X $ \triangleq $ Y'' stands for the idea that X is defined as Y, and ``$ \triangleq $'' is accordingly pronounced as ``is defined as'' or its likes. 

Let us define:
\begin{itemize}
  \item[D0:] $\hat{} \, yA \triangleq |_{\underline{0}} y A$
  \item[D1:] $\forall yA\triangleq |_{\underline{1}} y A$
  \item[D2:] $  \downarrow \! AB \triangleq |_{\underline{2}}AB$  
\item[D3:] $  \downarrow \! ab \triangleq |_{\underline{2}}ab$
 \item[D4:] $\mathrm{T}w\triangleq |_{\underline{3}}w$
  \item[D5:] $ $\euro$ w \hspace{2pt}\triangleq  |_{\underline{4}}w$
  \item[D6:] $v_{\underline{n}} \triangleq |_{\underline{n+5}}$
  \end{itemize}

With these notions we use Hebrew $\textcjheb{n}$ (Nun) in our definition of the set of noemata of expressions:
\begin{definition} \phantomsection\label{freevariables}

\begin{align*}
\textcjheb{n}(u) & \triangleq  \fat \{u\fat \} \\
\textcjheb{n}(\mathrm{T})&\triangleq \fat \{\fat\} \\
\textcjheb{n}(\mbox{\euro})&\triangleq \fat \{\fat\} \\
\textcjheb{n}(ba) &\triangleq \textcjheb{n} (a)\cup \textcjheb{n} (b) \\
\textcjheb{n} (\hspace{1pt} \hat{}\hspace{1,5pt}uA) &\triangleq \textcjheb{n}(A)\setminus\fat \{u\fat\} \\
\textcjheb{n} (\forall uA) &\triangleq \textcjheb{n}(A)\setminus\fat \{u\fat\} \\
\textcjheb{n}(\downarrow \hspace{-2pt}AB) &\triangleq \textcjheb{n} (A)\cup \textcjheb{n} (B) \\
\textcjheb{n}(\downarrow \hspace{-2pt}ab) &\triangleq \textcjheb{n} (a)\cup \textcjheb{n} (b) \\
\end{align*}

\end{definition}

A noema $u$ is \textit{present} in an expression $A$ iff $u\fat\in \textcjheb{n}(A)$. A formula $A$ is a \textit{proposition} iff no noema is present in $A$.  A formula $A$ is
\textit{atomic} iff $A$ is of the form $ba$ with terms $a$ and $b$. With this terminology, all propositions are sentences, so we do not presuppose that propositions are
extralinguistic entities in our framework. In \pounds \ it turns out that all formulas are \textit{parivalent} (see section \ref{Negjunction} below) with a proposition. 

A term 
$a$ is a \textit{cognomen} iff no noema is present in $a$. Some terms are neither cognomina nor praenomina. A term is a \textit{pronomen} iff it is a cognomen and neither $ \mathrm{T} $ (the alethizor) nor \euro \ (the enumerator) occurs in it. In \pounds \ we take a sort to be a  \textit{set} iff \pounds \ shows that it is identical with a pronomen.

For any term $a$ and noema $u$, $(a/u)$ is a substitution function from expressions to expressions:\bigskip

$(a/u)w\triangleq a$ if $ u=w $, otherwise $%
(a/u)w\triangleq w$

$(a/u)cb\triangleq(a/u)c(a/u)b$

$(a/u)\hspace{-3pt} \downarrow \hspace{-3pt}cb\hspace{1pt}\triangleq \hspace{-1pt}\downarrow \hspace{-2pt} (a/u)c(a/u)b$

$(a/u)\mathrm{T}\triangleq \mathrm{T}$

$(a/u) \mbox{\euro}\triangleq\mbox{\euro}$

$(a/u)\hspace{3pt}\hat{}\hspace{2pt}yA\triangleq \hat{}\hspace{2pt}y(a/u)A$ if $u\neq y$, else $(a/u)\hspace{2pt}\hat{}\hspace{2pt}yA\triangleq \hspace{1pt}\hat{}\hspace{1pt}yA$

$(a/u) (\forall y)A\triangleq (\forall y)(a/u)A$ if $u\neq y$, else 
$(a/u) (\forall y)A\triangleq (\forall y)A$

$(a/u) \hspace{-2pt} \downarrow\hspace{-2pt} AB\triangleq\hspace{-1pt} \downarrow\hspace{-2pt}(a/u)A(a/u)B$

\bigskip
We will make use of a suffix notation and write $A(a/u)$ for $%
(a/u)A$. Iterated uses of the substitution function like $A(a/u)(b/w)(c/y)$ should be written $A(a/u,b/w,c/y)$.

In the metalinguistic account we also presuppose the following definitions:

\begin{itemize}
   \item[D7:] $\lnot A\triangleq \downarrow \! AA$
  \item[D8:] $A\vee B\triangleq \downarrow \downarrow AB \! \downarrow AB$
  \item[D9:] $A\wedge B\triangleq \downarrow \downarrow AA \downarrow BB$
  \item[D10:] $A\rightarrow B\triangleq \downarrow \downarrow \downarrow AAB \downarrow \downarrow AAB$
  \item[D11:] $A\leftrightarrow B\triangleq (A\rightarrow B)\wedge (B\rightarrow A)$
  \item[D12:] $ \overline{a}  \triangleq \downarrow \! aa$
  \item[D13:] $a\cup b\triangleq \downarrow \downarrow ab \! \downarrow ab$
  \item[D14:] $a\cap b\triangleq \downarrow \overline{a} \overline{b}$
  \item[D15:] $a\setminus b\triangleq a\cap\overline{b}$
  \item[D16:] $\exists yA\triangleq \lnot \forall y\lnot A$
  \item[D17:] $a\hspace{-2pt}\in\hspace{-2pt} b\triangleq ba$
  \item[D18:] $\{y|A\} \triangleq \hat{} \, yA$
    \item[D19:] $\mathbb{T}A\triangleq (\exists y)(y\in\{y|A\})$, for $y\fat\notin\textcjheb{n}(A) \ \fat \& \ \fat\Pi m\fat (m \ is \ a \ noema  \ \fat \&  \ m$\fat \prec $y \  \fat\Rightarrow \ m\fat\in\textcjheb{n}(A)\fat ) $ \phantomsection\label{D19}
\end{itemize}

The metalinguistic connectives have the precedence order $ \lnot$,  $\wedge$,  $ \vee$, $ \rightarrow$ and then $ \leftrightarrow $. Parentheses are suppressed accordingly, and also associativity taken into account and outer parentheses omitted in order to better the legibility of the metalinguistic presentation. We call the $\downarrow$ the \textit{joint}, and the joint is a connective when operating on formulas and a \textit{juncture} when operating on terms; just junctures and connectives defined in terms of the joint are \textit{concourses}. As already stated, we use $u$, $v$, $w$, $ x $, $y$, $ z $, $u'$ etc. to range over noemata and variables in the metalinguistic account when availing ourselves of D(0)-D(19). When needed we use subscripts such as in $ x_{0} $ and $ x_{1}$. $ A(x,y,z) $ signifies that $ x $, $ y $ and $ z $ are noemata present in $ A $. We use $ \vv{x} $ in $ A(\vv{x})$  for $ A(x_{0},...,x_{n-1}) $ where $ n $ is unspecified for vectors of arbritrary dimension; the quantified expression $ \forall \vv{x}A$ is to be interpreted accordingly. Sentences and nomina that are presented (more or less) in line with these metalinguistic conventions are said to be in (more or less) \textit{presentable form}.

\section{Formations as Numbers, Mathematicalism and Coding} \label{Coding}
\label{Formations as numbers}

We adopt the policy of taking a formation of \pounds \ to $be$ the natural number (finite von Neumann ordinal) of \textbf{L}$_{\varsigma}$ denoted by the binary numeral which is gotten by replacing every occurrence of ``$ | $'' in its austere expression with ``$ 1 $'' and every occurrence of ``.'' in its austere expression with ``$ 0 $''.  Alternatively put, if one in the austere expression takes ``$ | $'' as the numeral ``$ 1 $'' and ``.'' as the numeral ``$ 0 $'' one may take the formation as the natural number denoted by it when taken as a binary numeral. The number corresponding to any symbol $ |_{\underline {n}}$ is $2^{n}$. The length of a natural number $n$\ in the binary system is $l(n)=\mu y(2^{y}>n)$. Define, for $ m\cdot n >0 $, $m\frown n\triangleq (m\cdot2^{l(n)}+n).$ $l(m \frown n)=l(m)+l(n)$ can be shown to hold by induction, and associativity then follows easily. For formations $ e $ and $ e' $ of \pounds, $ ee' $ is taken as $ e \frown e' $.

Our approach to formations as natural numbers is one which can be taken in general to formal and super (semi) formal systems. As an upshot the sets of theses of these systems may be understood as real numbers, i.e. sets of natural numbers. This supports a $mathematicalist$ point of view which holds that mathematics is more fundamental than logic. One may want to take the consequence relation of a logic as more fundamental than its theses or consequences from the empty set of assumptions. However that may be, from a mathematicalist point of view as suggested here the consequence relation of any given logic just \textbf{is} a function from real numbers to real numbers.


The coding function $\Godelnum{} \colon \mathbf{N} \to \mathbf{N}$ is given by the following recursion wherein ${}_{\bigtriangleup}$ is taken to abbreviate $|_{\underline{1}}|_{\underline{5}}|_{\underline{2}}|_{\underline{2}}|_{\underline{2}}|_{\underline{5}}|_{\underline{6}}|_{\underline{5}}|_{\underline{6}}|_{\underline{5}} {}_{{\scriptsize \Godelnum{n}}}|_{\underline{2}}|_{\underline{2}}|_{\underline{5}}|_{\underline{6}}|_{\underline{5}}|_{\underline{6}}|_{\underline{5}} {}_{{\scriptsize \Godelnum{n}}}$:\\

$\Godelnum{m} \triangleq \begin{cases}
|_{\underline{0}}|_{\underline{5}}|_{\underline{1}}|_{\underline{6}}|_{\underline{2}}|_{\underline{2}}|_{\underline{2}}|_{\underline{6}}|_{\underline{5}}|_{\underline{6}}|_{\underline{5}}|_{\underline{6}}|_{\underline{5}}|_{\underline{2}}|_{\underline{2}}|_{\underline{6}}|_{\underline{5}}|_{\underline{6}}|_{\underline{5}}|_{\underline{6}}|_{\underline{5}} &\mbox{if} \ m=.\\
|_{\underline{0}}|_{\underline{5}}|_{\underline{2}}|_{\underline{2}} {}_{{\scriptsize \Godelnum{n}}}|_{\underline{5}}{}_{\bigtriangleup}|_{\underline{2}}{}_{{\scriptsize \Godelnum{n}}}|_{\underline{5}} {}_{\bigtriangleup} & \mbox{if} \ m=n+1
\end{cases}$\\

\smallskip
\medskip

\begin{remark}
Recall that we adopt the numerical policy of taking ``$|$'' and ``.'' as the numerals ``1'' and ``0'' of the binary number system denoting finite von Neumann ordinals of $\mathbf{L}_{\varsigma}$, so that we take expressions of \pounds \ to \textbf{be} natural numbers of \textbf{L}$_{\varsigma}$.
\end{remark}
\begin{remark}
 $\Godelnum{.}$ is stated in bare form and $\Godelnum{n+1}$ in bare form modulo $\Godelnum{n}$; it is left as an exercise to work these out in more presentable forms given definitions D0-D19 above.
\end{remark}

\begin{remark}\label{codingremark}
Our coding function and numerical policy are so devised that for any formation $\mathbf{f}$ of \pounds \ a natural number is $\mathbf{f}$ according to \textbf{L} iff it is  $\Godelnum{\mathbf{f}}$ according to \pounds. In light of our policy laid down above concerning expressions of \pounds \ as numbers of \textbf{L}$_{\varsigma}$, this means that whenever $\mathbf{ef}$ is a formation of \pounds, then $\Godelnum{\mathbf{ef}}$ is $\Godelnum{\mathbf{e}\frown \mathbf{f}}$; with the numerical policy in mind, the following definition is useful.
\end{remark}

\begin{definition}
$\Godelnum{\mathbf{e}}\Godelnum{\frown}\Godelnum{\mathbf{f}} \triangleq \Godelnum{\mathbf{e}\frown \mathbf{f}}=\Godelnum{\mathbf{e}\cdot2^{l(\mathbf{f})}+\mathbf{f}}$\label{Concatenation}
\end{definition}

\begin{remark}
The central idea for the base case is to find the smallest number of \textbf{N} of \textbf{L} that denotes 0 according to \pounds \ without invoking noemata. In more presentable form $\ulcorner . \urcorner$ corresponds with $\{ v_{\underline{0}}| \lnot \forall v_{\underline{1}} (v_{\underline{0}} \in v_{\underline{1}} \rightarrow v_{\underline{0}} \in v_{\underline{1}}    ) \}$, but notice that the bare or austere form is strictly speaking needed to be exact and that the choice of variables is not arbitrary; to the last point, cfr. remark \ref{remark7} just below. If my count is right $\ulcorner . \urcorner$ is about the size of $ 2^{111}$. Work the opposite way for the successor case so that the definiens of $\ulcorner n+1 \urcorner$ in more presentable form is $\{v_{\underline{0}}| v_{\underline{0}} \in \ulcorner n \urcorner \vee \forall v_{\underline{1}} (v_{\underline{0}} \in v_{\underline{1}} \rightarrow \ulcorner n \urcorner \in v_{\underline{1}}   )\}$. Obviously, one could presuppose other choices for zero and successor.
\end{remark}
\begin{remark}
 In the definientia of $\ulcorner n+1 \urcorner$ and of $\ulcorner .\urcorner$=$\ulcorner 0 \urcorner$ as explicated in the previous remark the Leibnizian-Russellian theory of identity justified in section 4 of \citep{Frode2012a} is presupposed.
\end{remark}
\begin{remark}\label{remark7}
Notice that $\Godelnum{.}$ is a different number from $|_{\underline{0}}|_{\underline{8}}|_{\underline{1}}|_{\underline{7}}|_{\underline{2}}|_{\underline{2}}|_{\underline{2}}|_{\underline{7}}|_{\underline{8}}|_{\underline{7}}|_{\underline{8}}|_{\underline{7}}|_{\underline{8}}|_{\underline{2}}|_{\underline{2}}|_{\underline{7}}|_{\underline{8}}|_{\underline{7}}|_{\underline{8}}|_{\underline{7}}|_{\underline{8}}$ in \textbf{L}, but $\cVvdash \Godelnum{.}=|_{\underline{0}}|_{\underline{8}}|_{\underline{1}}|_{\underline{7}}|_{\underline{2}}|_{\underline{2}}|_{\underline{2}}|_{\underline{7}}|_{\underline{8}}|_{\underline{7}}|_{\underline{8}}|_{\underline{7}}|_{\underline{8}}|_{\underline{2}}|_{\underline{2}}|_{\underline{7}}|_{\underline{8}}|_{\underline{7}}|_{\underline{8}}|_{\underline{7}}|_{\underline{8}}$ is a maxim of \pounds \ given our policy on alphabetological variants (see Section \ref{alphabetological} on this notion) as per principle $P6$ in the semantic recursion in Section \ref{The semantics of pounds} below. So the notion of identity of \pounds \ does not coincide with that of \textbf{L}.
\end{remark}

\section{Substitution}\phantomsection\label{substitution}

We define a substitution function on triples of natural numbers by a course of value recursion:

\medskip

sub$(\ulcorner v_{\underline{i}} \urcorner, i, y )$=$y$
\smallskip

sub$(\ulcorner v_{\underline{i}} \urcorner, j, y )$=$\ulcorner v_{\underline{i}} \urcorner$ if $i\neq j$

\smallskip

sub$(\ulcorner ab \urcorner, i, y)$=sub$(\ulcorner a \urcorner, i, y)\frown$sub$(\ulcorner b \urcorner, i, y)$ for atomic formula $ab$

\smallskip

sub$(\ulcorner$$\downarrow$$  AB \urcorner, i, y)$ = $\ulcorner$$\downarrow$$  \urcorner \frown$ sub $(\ulcorner A \urcorner, i, y) \frown$ sub $(\ulcorner B \urcorner, i, y)$

\smallskip

sub$(\ulcorner \forall v_{\underline{i}}A\urcorner, i, y)$=$\ulcorner \forall v_{\underline{i}}A\urcorner$

\smallskip
sub$(\ulcorner \forall v_{\underline{j}}A\urcorner, i, y)$=$\ulcorner \forall \urcorner \frown \ulcorner v_{\underline{j}} \urcorner \frown$ sub $(\ulcorner A \urcorner, i, y)$ if $i \neq j$
\smallskip

sub$(\ulcorner \, \hat{} \, v_{\underline{i}}A\urcorner, i, y)$=$\ulcorner \, \hat{} \, v_{\underline{i}}A\urcorner$
 
\smallskip
sub$(\ulcorner \, \hat{} \,v_{\underline{j}}A\urcorner, i, y)$=$\ulcorner \, \hat{} \, \urcorner \frown \ulcorner v_{\underline{j}} \urcorner \frown$ sub $(\ulcorner A \urcorner , i, y)$ if $i \neq j$

sub$(x, i, y)$=\hspace{2pt}x if x is not of one of the above forms.

\medskip
We see that for any term $t$ and formula $A(v_{\underline{i}})$, sub$(\ulcorner A(v_{\underline{i}}) \urcorner, i, \ulcorner t \urcorner)$=$\ulcorner A(t)\urcorner$. Compare here and in the following \citep{Smorynsky1977}, p. 837.

We define Sub(x,y)=sub(x, i, y) \phantomsection\label{Sub} whenever x is the code of a formula with a noema and $i$ is the least number such that noema $ v_{\underline{i}}$ occurs in the formula x is a code of, and Sub(x,y)=x if x is not the code of a formula with a noema. We define {\scriptsize SUB}(x,y)=Sub(x,$ \ulcorner y \urcorner $) and abbreviate {\scriptsize SUB}($\ulcorner A(x) \urcorner, y$ ) by $\ulcorner A(\dot{y}) \urcorner$. Iterated uses are as expected, and if the vector's dimensionality is unspecified we write $\Godelnum{ A(\dot{\vv{x}})}$. 

\section{Alphabetological Variants and their Enumeration}\label{alphabetological}

Let $\mathrlap{\mathbf{T}}\mathbf{L} (A \leftrightarrow B)$ hold in \textbf{L} iff $A$ and $B$ are propositions (``closed sentence'') according to classical logic in the language of \pounds \ which are provably equivalent in classical logic, i.e. they are \textit{Lindenbaum-Tarski congruent}. Proposition $A$ is an \textit{alphabetological variant} of proposition $B$ iff $\mathrlap{\mathbf{T}}\mathbf{L} (A \leftrightarrow B )$ or there are cognomina (i.e. terms wherein no noema occurs) $a$ and $b$ which are alphabetological variants of each other and such that $ A $ is an alphabetological variant of $B(a/b)$. Cognomina $a$ and $b$ are alphabetological variants of each other iff $a=b=\mathrm{T}(=|_{\underline{3}})$ or $a=b=$\euro$ \hspace{2pt}(= |_{\underline{4}})$, or there are cognomina $c$ and $d$ and $e$ and $f$ such that $a=\downarrow \hspace{-3pt}cd$  and $b=\downarrow \hspace{-3pt}ef$ and either $c$ is an alphabetological variant of $e$ and $d$ is an alphabetological variant of $f$ or $c$ is an alphabetological variant of $f$ and $d$ is an alphabetological variant of $e$, or there are formulas $A$ and $B$ so that $a=\{x:A(x)\}$ and $b=\{y:B(y)\}$ and $( \exists x)A(x)$ is an alphabetological variant of $ (\exists y) B(y)$.


Cognomina of \pounds, as other formations of \pounds,  have a natural order according to their sizes as natural numbers or finite von Neumann ordinals of $\textbf{L}_{\varsigma}$.  We presuppose a Kuratowskian definition of ordered pairs here and use the following notation to distinguish object language and meta language statements involving ordered pairs: $ \langle a, b \rangle$ is an object language ordered pair of \pounds \ and $\fat < \hspace{0pt} a, b\hspace{0pt} \fat >$ is an ordered pair of $\textbf{L}_{\varsigma}$. We define: 

\begin{definition}
$\fat <  n, b \fat > \fat \in \textbf{e}$ $\triangleq$ $n\fat\in\mathbf{N}$ and $b$ is a cognomen and an alphabetological variant $c$ of $b$ is the smallest cognomen such that $\fat\sim \fat \Sigma m\fat (m\fat\in\mathbf{N} \ \fat \& \ m \fat \prec n \ \fat \& \ \fat < m,c\hspace{0pt} \fat > \fat{\in} \textbf{e}\fat )$.
\end{definition}

\section{The Semantics of \pounds} \label{The semantics of pounds}

We first define the notion `$a$ \textit{is substitutable for} $u$ 
\textit{in} \ldots '\ for terms $a$ and noemata $u$ by the recursion: $a$ is substitutable for $u$ in $ y $ when $y$ is a noema; $a$ is substitutable for $u$ in $cb$
iff $a$ is substitutable for $u$ in $b$ and in $c$; $a$ is substitutable for $u$ in $ \downarrow\hspace{-2pt} bc $ iff $a$ is substitutable for $u$ in $b$ and in $c$; $a$ is substitutable for $u$ in $ \downarrow\hspace{-2pt} AB $ iff $a$ is substitutable for $u$ in $A$ and in $B$; $a$ is
substitutable for $u$ in $\hat{}\hspace{2pt} y A$ iff either $y$ is not a noema in $a$ or $u$ is not a noema in $A$ and $a$ is substitutable for $u$ in $A$; $a$ is substitutable
for $u$ in $\forall yA$ iff either $y$ is not a noema in $a$ or $u$ is not a noema in $A$ and $a$ is substitutable for $u$ in $A$. We have here adapted the account of \cite{enderton:2001a} p. 113. We also say that $a$ is free for $u$ in an expression iff $a$ is substitutable for $u$ in the expression, and write $\fat F(a,u,A)$ when the expression is a formula $A$ and $\fat F(a,u,b)$ when the expression is a term $b$.

In stating the semantical principles below we use almost presentable form. 
In defining the semi inductive Herzbergerian style semantic process (cfr.  \cite{Gupta1982-GUPTAP} and descendent literature for the related \textit{revisionary} approach)  we use the syntax specified at the beginning of Secion \ref{Fundamentals}. $\mathfrak{F}$ is the arithmetical predicate of $\textbf{L}_{\varsigma}$ for formulas of \pounds, and $\mathfrak{T}$ is its arithmetical predicate for terms of \pounds. We avail ourselves of the class of von Neumann ordinals in the transfinite recursion, and let minuscule Greek letters $\fat \alpha$, $\fat \beta$, ... plus the archaic $\fat \Koppa $ (koppa) denote ordinals and let $ m $ and $ n $ range over natural numbers. In \texttt{P}$5$, \texttt{P}$6$ and \texttt{P}$7$ we presuppose the Leibnizian-Russellian definition shown adequate in \citep{Frode2012a} section 4, so $ a=b\triangleq\forall u (a\in u \rightarrow b \in u) $. $\vDash$ is a function from ordinals to real numbers taken as sets of natural numbers of $\textbf{L}_{\varsigma}$, and we write $\fat \alpha\hspace{-2pt} \vDash\hspace{-2pt} A$ for $A \fat \in \hspace{-2pt}\vDash\hspace{-3pt} (\hspace{-1pt}\fat \alpha\hspace{-1pt})$. We presuppose that formulahood is fulfilled. For  $\mathbb{T}$ in \texttt{P}$5$ and below, recall D19 on page \pageref{D19}.


\medskip

\noindent \texttt{P}$1$: $\fat\alpha \vDash a \in \{u|A\}  \Leftrightarrow $ \fat F(a,u,A) \fat \& {\footnotesize $\fat\Sigma \fat \beta \fat (\fat \beta \fat \prec\fat \alpha \ \fat \&\  \fat \Pi \fat \gamma \fat (\fat \beta \fat \preceq \fat \gamma \fat \prec\fat\alpha \fat\Rightarrow \fat\gamma \vDash A(a/u)\fat )\fat )$}\phantomsection\label{P1}\\ 
\texttt{P}$2$: $\fat \alpha \vDash \downarrow\hspace{-2pt} AB \Leftrightarrow$ $\fat\alpha \nvDash A$ \& $\fat\alpha \nvDash B$ \\
\texttt{P}$3$: $\fat\alpha \vDash a \in \downarrow\hspace{-2pt} bc \Leftrightarrow$ $\fat\alpha \nvDash a\in b$ \& $\fat\alpha \nvDash a\in c$\\
\texttt{P}$4$: $\fat\alpha \vDash \forall y A \Leftrightarrow \fat \Pi a\fat (\fat F(a,y,A)\Rightarrow \fat\alpha \vDash A(a/y)\fat )$\\
\texttt{P}$5$: $\fat\alpha \vDash \mathrm{T}u \Leftrightarrow \fat\Sigma A \fat (\mathfrak{F}(A) \fat \& \fat\alpha \vDash u=\Godelnum{A}\wedge\mathbb{T} A\fat )$ \\
\texttt{P}$6$: For $\fat < n,a\fat > \fat \in \mathbf{e}$: If {\small $\fat\Pi  m \fat (m\fat \prec n \ \fat \Rightarrow \ \fat \alpha \vDash v_{\underline{m}} \neq a\fat ) $, $\fat \alpha \vDash v_{\underline{n}}=a$};  else $\fat \alpha \vDash v_{\underline{n}}=v_{\underline{n-1}}$\\
\texttt{P}$7$: $\fat\alpha \vDash \hspace{-2pt} u \hspace{-3pt}\in$\hspace{-3pt}  \euro \hspace{2mm}$\Leftrightarrow \fat \Sigma a \fat \Sigma n \fat (a\fat\in\mathbf{N}\fat \& n\fat\in\mathbf{N}\fat \& \mathfrak{T}(a)\fat \& \fat\alpha \vDash u=\langle \Godelnum{n},a \rangle  \wedge a=v_{\underline{n}}\fat ) $

\bigskip

\noindent We define: \medskip

\begin{definition}\phantomsection\label{6.1}
$\vDash _{\fat\gamma} \triangleq \fat \{A\fat | \fat \gamma \vDash \mathbb{T}A \fat \}$

\end{definition}

\begin{definition}
$IN_{\vDash }\triangleq \fat \{A\fat |\mathfrak{F}(A)\  \fat \& \ \fat \Sigma \fat \beta \fat \Pi \fat \gamma \fat (\fat \beta \fat \preceq \fat \gamma
\fat \Rightarrow A\fat\in \vDash_{\fat \gamma}\fat )\fat \}$
\end{definition}

\begin{definition}
$OUT_{\vDash }\triangleq \fat \{A\fat | \mathfrak{F}(A)\ \fat \& \ \fat \Sigma \fat \beta \fat \Pi \fat \gamma \fat (\fat \beta \fat \preceq \fat \gamma
\fat \Rightarrow A\fat\notin \vDash_{\fat \gamma}\fat )\fat \}$
\end{definition}
 
\begin{definition}
$STAB_{\vDash }\triangleq IN_{\vDash}\hspace{1pt}\fat\bigcup\hspace{1pt} OUT_{\vDash }$
\end{definition}

\begin{definition}
$UNSTAB_{\vDash }\triangleq \fat \{A\fat|\mathfrak{F}(A)\fat \}\backslash STAB_{\vDash }$
\end{definition}

\begin{definition}
Limit $\fat\kappa $ \textit{covers } $\vDash $ $ \triangleq $ $\fat \Pi \fat \gamma\fat (\fat \kappa \fat \preceq \fat \gamma\fat \Rightarrow \fat (IN_{\vDash }\hspace{3pt}\fat\subset \hspace{-1pt} \vDash _{\fat\gamma}\fat \&\vDash _{\fat\gamma}\hspace{-1pt}\fat\subset \hspace{1pt}IN_{\vDash}\hspace{1pt}\fat\bigcup\hspace{1pt} UNSTAB_{\vDash }\fat ) \fat )$
\end{definition}

\begin{definition}
Limit $\fat \sigma $ \textit{stabilizes }$\vDash $ $ \triangleq $ $\fat \sigma $ covers $\vDash $ and 
$\vDash_{\fat \sigma}\ \subset IN_{\vDash}$
\end{definition}

\begin{theorem}[Herzberger]\ \label{Herzberger}

\begin{enumerate}
  \item[(i)] There is an ordinal $\fat\kappa$ which covers $\vDash$.	
  \item[(ii)] There is an ordinal $\fat\sigma$ which stabilizes $\vDash$.
\end{enumerate}

\end{theorem}

\begin{proof}[Proof:] (i) By definition, $\fat \Pi A(A\fat \in IN_{\vDash }\Rightarrow\fat \Sigma \beta \fat \Pi \gamma (\beta \fat\preceq \gamma
\Rightarrow A\fat \in \hspace{-4pt}\vDash_{\gamma}))$. This is equivalent to $\fat \Pi A\fat\Sigma\beta\fat\Pi\gamma(A\fat \in IN_{\vDash }\Rightarrow (\beta \fat\preceq \gamma
\Rightarrow A\fat \in \hspace{-4pt}\vDash_{\gamma}))$. As $\mathbf{L}_{\varsigma}$ is closed under sufficient collection we have that
 $\fat \Pi A(A\fat \in IN_{\vDash }\Rightarrow \fat\Sigma\beta\fat\in a\fat \Pi \gamma (\beta \fat\preceq \gamma
\Rightarrow A\fat \in \hspace{-4pt}\vDash_{\gamma}))$ for some $a$. The ordinal $\nu =\fat\bigcup\fat \{x\fat |x\fat\in a \fat \& Ord(x) \fat\}$ is a covering ordinal. 
We gauge the amount of collection needed by using the definition of $IN_{\vDash}$ to see that we have invoked $\Pi_{2}$-collection, and as $\Pi_{n}$-collection implies $\Sigma_{n+1}$ collection in the context of Kripke Platek set theory, this justifies our choice of $\mathbf{L}_{\varsigma}$ with $\varsigma$ as the least $\Sigma_{3}$-admissible ordinal. (ii) Let $\delta $ be the least ordinal which covers $\vDash$.
Let $\fat\{f(n)\fat |n\fat \in \fat\omega \fat\}$ by an adaptation of Cantor's pairing function be an enumeration of all
elements of $UNSTAB_{\vDash }$ where each element recurs infinitely often so that if $%
B\textbf{=}f(m)$ and $m \fat \prec n\fat \in \omega $, then there is a natural
number $n^{\prime }$, $n \fat \prec n^{\prime }\fat \in \omega ,$ such that $f(n^{\prime
})=B$. Let $F(0)=\delta $ and $F(n+1)=$
the least $\nu >F(n)$ such that $f(n)\fat \in \vDash(\nu )$ iff $f(n)\hspace{0pt}\cancel{\fat \in} \vDash(F(n))$.
Let $\fat \phi =\fat \{\gamma \fat |\fat \Sigma m\fat \Sigma \nu (m\fat \in \omega $ $\&$ $\nu
=F(m) $ $\&$ $\gamma \fat \in \nu )\fat \}$. It is obvious that $\fat \phi $ is a limit
ordinal which covers $\vDash$. It is also clear that if $m\fat \prec n\fat \in \omega $
then $F(m)\fat \prec F(n)$. Since $\fat \phi $ covers $\vDash$, it suffices to show
that $B\fat \in\hspace{-2pt} \vDash_{\fat \phi }$ entails that $B\fat \in STAB_{\vDash }$ to establish that $\fat \phi $ stabilizes $\vDash$. Suppose $B\fat\in\hspace{-2pt} \vDash_{\fat \phi }$. By Definition \ref{6.1} on page \pageref{6.1}, then $\fat \phi\vDash\mathbb{T}B$. By definition D19 of page \pageref{D19}, $\fat \phi\vDash\exists x(x\in\{x|B\})$ for $x\fat\notin\textcjheb{n}(B)$ (cfr. Definition \ref{freevariables} on page \pageref{freevariables}). By $\mathrm{P}1$-$4$ of page \pageref{P1}, there exists an ordinal $\nu$ so that:

\medskip a)\qquad $\fat \Pi \mu \fat (\nu \fat\preceq \mu \fat \prec\fat \phi \Rightarrow 
\mu\vDash B\fat ) $ \medskip

Since $F$ is increasing with $\fat \phi $ as its range, we will for some
natural number $m\fat \in \omega $ have that $\nu \fat\preceq F(m)\fat \prec\fat \phi $, so that
\medskip

b)\qquad $\fat \Pi \mu \fat (F(m)\fat\preceq \mu \fat\prec\fat \phi \Rightarrow  \mu\vDash B\fat )
$

\vspace{3pt}

Suppose $B\fat \notin STAB_{\vDash }$. By our enumeration of
unstable elements where each term recurs infinitely often we have that 
$B=f(n)$ for some natural number $n$, $m\fat\prec n\fat\in\fat \omega $.
It follows that $F(m)\fat \prec F(n)\fat\prec\fat \phi $. From a) and b) we can then infer that $F(n)\vDash B$, since we have supposed that $\fat \phi\vDash B$. But from the construction of the function $F$ it
would then follow that $F(n+1)\not\vDash B$,
contradicting b). It follows that $B\fat\in\hspace{-2pt}\vDash_{\fat \phi}$ only
if $B\fat\in STAB_{\vDash }$, so that $\fat \phi $ stabilizes $\vDash $.   \end{proof}

\begin{lemma}
The least stabilizing ordinal is countable.
\end{lemma}

Compare the antecedent and now superseded account of \cite{Frode2012a}, and relatedly also \cite{Herzberger1980}, \citep{Herzberger1982}, \cite{Cantini1996}, \cite{Welch2003} and elsewhere. The semantic set up we provide here in important respects deviates from and simplifies arguments given earlier. Our rather pedantic semantic machinery has an important pay off as we may realize that the \textit{closure ordinal} (i.e. the least stabilising ordinal) is countable without any hint of appeal to anything uncountable as with L\"{o}wenheim Skolem considerations, as no set is uncountable according to $\Sigma_{3}KP$. 
Notice also that we now presuppose an extension of the formal language which also gives us the real complement of sets, so that at the initial ordinal $0$ for all terms $b$ of \pounds, $0\vDash\forall x(x\notin b)$ or $0\vDash\forall x(x\in b)$. 

We let \fat \Koppa\ - archaic Greek Koppa - be the closure ordinal. We have that $ \fat \Koppa \vDash \mathbb{T}A $  iff for all $ \fat\beta\fat\succeq\fat \Koppa $, $ \fat\beta \vDash A $.

We make the crucial \emph{librationist twist} to isolate the
intended model of librationism, and shift our attention to those formulas (sentences) $A$ which are such that $\fat \Koppa \vDash
\lnot \mathbb{T}\lnot A$. So our official definition of what we take as the \textit{roadstyle} sign
is $\cvdash A\triangleq \fat \Koppa \vDash \lnot \mathbb{T}\lnot A$.

All ordinals $ \fat\alpha $ are \textit{maximally progression consistent} in the sense that $%
\fat\alpha \vDash B$ iff $\fat\alpha \nvDash \lnot B$, and \textit{progression closed} in the sense that $\fat\alpha \vDash A$ and $\fat\alpha \vDash A \rightarrow B$ only if $\fat\alpha \vDash B$. Suppose ${\cancel \cvdash} A$, i.e. not $\cvdash A$. It follows that $\fat \Koppa \vDash \mathbb{T}\lnot A$ as $ \fat \Koppa $ is maximally progression consistent. But $%
\fat \Koppa \vDash \mathbb{T}\lnot A\rightarrow \lnot \mathbb{T} A$ on account of $\pounds \aoverb {2}{M}$ below. As $\fat \Koppa$ is progression closed, it follows that $\fat \Koppa\vDash \lnot \mathbb{T}A$, i.e. $\cvdash \lnot A$.
So $\cvdash A$ or $\cvdash \lnot A$, i.e. \pounds \ is negjunction (``negation'') complete.

Our definition of the roadstyle supports the following precise definitions of maxims (signified with $\cVvdash$) and minors (signified with $\cVdash$): $\cVvdash  A\triangleq \fat \Koppa\vDash \mathbb{T}A$ and $\cVdash A\triangleq\hspace{2pt}\cvdash A$ $\&$ $\cvdash \lnot
A.$ As $ \fat \Koppa $ is  progression consistent we have that $ \cVvdash A $ iff both $\cvdash A  $ and  ${\cancel \cvdash} \lnot A $. As \pounds \ is negjunction complete, $ \cVvdash A $ iff ${\cancel \cvdash} \lnot A $. Intuitively, the maxims are the unparadoxical theses while the minors are the paradoxical ones.

\section{Prescriptions, Prescripts and Prescribes as Posits of \pounds}
\label{prescriptions}

We give a partial list of posits of \pounds. All prescriptions that follow hold with all generalizations, so that generalization is
not a primitive regulation. We can show, however, by an inductive
argument going back to Tarski, that generalization holds as a derived
regulation relative to theses which follow from the prescriptions presupposed with all generalizations. Maximal prescriptions are those which only have maxims as instances, and they are marked with a subscripted capital M in their appellation. All instances of minor prescriptions (which are marked with a subscripted minuscule m in their appellation) are theses, but some of their instances are minors.

\begin{align*}
\pounds^i_\mathrm{M}   & \qquad A\rightarrow (B\rightarrow A) \\
\pounds^{ii}_\mathrm{M} & \qquad (A\rightarrow (B\rightarrow C))\rightarrow ((A\rightarrow B)\rightarrow (A\rightarrow C)) \\
\pounds^{iii}_\mathrm{M}  & \qquad (\lnot B\rightarrow \lnot A)\rightarrow (A\rightarrow B) \\
\pounds^{iv}_\mathrm{M}  & \qquad A\rightarrow \forall xA, \ \text{provided }x\fat\notin\textcjheb{n}(A) \text{} \\
\pounds^{v}_\mathrm{M} & \qquad \forall x(A\rightarrow B)\rightarrow (\forall xA\rightarrow \forall xB)  \\
\pounds^{vi}_\mathrm{M}  & \qquad \forall xA\rightarrow A(a/x), \ \text{if }\fat F(a,x,A)  \\
\pounds^1_\mathrm{M} & \qquad \mathbb{T}(A\rightarrow B)\rightarrow (\mathbb{T}A\rightarrow \mathbb{T}B)  \\
\pounds^2_\mathrm{M}  & \qquad \mathbb{T}A\rightarrow \lnot \mathbb{T}\lnot A \\
\pounds^3_\mathrm{M}  & \qquad \mathbb{T}B\vee \mathbb{T}\lnot B\vee (\mathbb{T}\lnot \mathbb{T} \lnot A\rightarrow \mathbb{T}A) \\
\pounds^4_\mathrm{M}  & \qquad \mathbb{T}B\vee \mathbb{T}\lnot B\vee (\mathbb{T}A\rightarrow \mathbb{TT}A) \\
\pounds^5_\mathrm{M}  & \qquad \mathbb{T}(\mathbb{T}A\rightarrow A)\rightarrow \mathbb{T}A\vee\mathbb{T}\lnot A \\
\pounds^6_\mathrm{M}  & \qquad \exists x\mathbb{T}A\rightarrow \mathbb{T}\exists xA \\
\pounds^7_\mathrm{M}  & \qquad \mathbb{T}\forall xA\rightarrow \forall x\mathbb{T}A \\
\pounds^8_m  & \qquad \mathbb{T}A\rightarrow A \\
\pounds^9_m  & \qquad A\rightarrow \mathbb{T}A \\
\pounds^{10}_m  & \qquad \forall x\mathbb{T}A\rightarrow \mathbb{T}\forall xA \\
\pounds^{11}_m  & \qquad \mathbb{T}\exists xA\rightarrow \exists x\mathbb{T}A
\end{align*}

The \textit{alethic comprehension} prescription:\bigskip

$%
\begin{array}{ll}
\mathrlap{\mathbf{C}}A_{\mathrm{M}}   & \forall x(x\in \{y|A\}\leftrightarrow \mathbb{T}A(x/y))\text{, if }x\text{
is substitutable for }y\text{ in }A\text{}%
\end{array}%
$
\bigskip

The \textit{truth} prescription:\bigskip

$%
\begin{array}{ll}
\mathfrak{T}_{\mathrm{M}}  & \forall \vv{x}(\mathbb{T}A(\vv{x}) \leftrightarrow \mathrm{T}\hspace{1pt}\Godelnum{ A(\dot{\vv{x}})})
\end{array}%
$
\bigskip

\begin{definition}\label{Kind}
$KIND(a)\triangleq\forall x(\mathbb{T}x\in a\vee\mathbb{T}x\notin a)$
\end{definition}

The \textit{enumeration} precripts:\bigskip

$
\begin{array}{ll}
\vspace{2pt}
\mbox{\euro}^{1}_{\mathrm{M}} & \forall x\exists y(y\in \mathbb{N}\wedge <y,x>\in\mbox{\euro})\\
\vspace{2pt}
\mbox{\euro}^{2}_{\mathrm{M}} & \forall x(x \in \mbox{\euro} \rightarrow \exists y,z(x=<y,z>))\\
\vspace{2pt}
\mbox{\euro}^{3}_{\mathrm{M}} & \forall x,y,z(<x,y>\in\mbox{\euro}\ \wedge<x,z>\in\mbox{\euro}\rightarrow y=z)\\
\vspace{2pt}
\mbox{\euro}^{4}_{\mathrm{M}} & \forall x,y,z(<x,y>\in\mbox{\euro}\ \wedge<z,y>\in\mbox{\euro}\rightarrow x=y)\\
\mbox{\euro}^{5}_{\mathrm{M}} & KIND(\mbox{\euro})
\end{array}
$

\bigskip

The \textit{prescribes} prescription:\medskip

$
\begin{array}{ll}
\mathrm{P}_{\mathrm{M}} & <n,v_{\underline{n}}>\in\mbox{\euro}$, for $n\in\mathbb{N}$ and corresponding $v_{\underline{n}}\fat \in\mathbf{N}
\end{array}
$

\bigskip
The \textit{disunion} prescript:\medskip

$
\begin{array}{ll}
\downarrow_{\mathrm{M}} & \forall x,y,z(x\in \downarrow \hspace{-2pt}yz \leftrightarrow x\notin y \wedge x\notin z)
\end{array}
$
\medskip

The \textit{complement} prescript:\medskip

$
\begin{array}{ll}
\overline{_{\mathrm{M}}} & \forall x, y(x\in \overline{y}\leftrightarrow x\notin y)
\end{array}
$

\medskip

The \textit{relative complement} prescript:\medskip

$
\begin{array}{ll}
\setminus_{\mathrm{M}} & \forall x, y, z(x\in y\setminus z \leftrightarrow x\in y \wedge x\notin z)
\end{array}
$

\medskip 

The \textit{union} prescript: \medskip

$
\begin{array}{ll}
\cup_{\mathrm{M}} & \forall x, y, z(x\in y\cup z \leftrightarrow x\in y \vee x\in z)
\end{array}
$

\medskip

The \textit{intersection} prescript: \medskip

$
\begin{array}{ll}
\cap _{\mathrm{M}} & \forall x, y, z(x\in y\cap z \leftrightarrow x\in y \wedge x\in z)
\end{array}
$

\begin{definition}
Let $\mathrm{F}\triangleq \mathfrak{F}\setminus \mathrm{T}$
\end{definition}

The \textit{bivalence} prescript: \medskip

$
\begin{array}{ll}
\mathrm{B}_{\mathrm{M}} & \forall x(\mathfrak{F}x\rightarrow\mathrm{F}x\vee\mathrm{T}x)
\end{array}
$
\medskip

I leave it as an exercise here to verify the prescriptions and prescripts above. But see \citep{Frode2012a} for accounts of a number of prescriptions, some of which are lifted from \cite{Cantini1996}. 

There are infinitely many distinct prescribes which contain at least one noema. Importantly, prescribes are not instances of prescripts and cannot be universally generalized upon. Nevertheless, by the prescribes prescription fundamental prescribes are instances of a prescription. As $a=b \triangleq \forall u (a \in u \rightarrow b \in u)$, $ v_{\underline{0}}=\mathrm{T}$ is an example of a prescribe relative to the presupposed enumeration \textbf{e} as $\mathrm{T}$ is the first cognomen in $ \textbf{e} $.

\textit{Librationist comprehension} is constituted by those principles which are engendered by alethic comprehension in cooperation with all posits and regulations of \pounds. 

\section{Regulations of \pounds}\label{Regulations}
The regulations valid in librationism are sensitive as to whether the initial or consequential theses are maxims or minors. We give a partial list of some salient regulations: \bigskip

\noindent $%
\begin{array}{lll}
R_{1} & \cVvdash A\text{ }\&\text{ }\cVvdash(A\rightarrow B)\Rightarrow 
\text{ }\cVvdash B & \text{\textit{modus maximus}} \\ 
R_{2} & \cVdash A\text{ }\&\text{ }\cVvdash (A\rightarrow B)\Rightarrow 
\text{ }\cvdash \hspace{0pt}B & \text{\textit{modus\ subiunctionis}} \\ 
R_{3} & \cVvdash A \hspace{3pt}\& \hspace{4.7pt}\cVdash(A\rightarrow B)\hspace{-0.8pt}\Rightarrow 
\text{ }\hspace{0pt}\cVdash \hspace{-0.3pt}B & \text{\textit{modus\ antecedentiae}} \\ 
R_{4} & \cVvdash A\Rightarrow \text{ }\cVvdash \mathbb{T}A & \text{%
\textit{modus\ ascendens\ maximus}} \\ 
R_{5} & \cVdash A\Rightarrow \text{ }\cVdash\mathbb{T}A & \text{%
\textit{modus\ ascendens\ minor}} \\ 
R_{6} & \cVvdash \mathbb{T}A\Rightarrow \text{ }\cVvdash A & \text{%
\textit{modus descendens maximus}} \\ 
R_{7} & \cVdash\mathbb{T}A\Rightarrow \text{ }\cVdash A & \text{%
\textit{modus descendens minor}} \\ 
R_{8} & \cVvdash \lnot \mathbb{T}\lnot A\Rightarrow \text{ }\cVvdash\mathbb{T}A & \text{\textit{modus\ scandens\ maximus}} \\ 
R_{9} & \cVdash\lnot \mathbb{T}\lnot A\Rightarrow \text{ }\cVdash\mathbb{T}A & \text{\textit{modus scandens minor}} \\ 
R_{10} & \cVvdash \forall x\mathbb{T}A\hspace{0pt}\Rightarrow \text{ }\cVvdash %
\mathbb{T}\forall xA & \text{\textit{modus Barcanicus}} \\ 
R_{11} & \cvdash \hspace{1pt} \mathbb{T}\exists xA\hspace{1.3pt}\Rightarrow \text{ }\cvdash\hspace{0pt} \exists x%
\mathbb{T}A & \text{\textit{modus attestans generalis}} \\ 
R_{12} & \cVdash\mathbb{T}\exists xA\hspace{1.3pt}\Rightarrow \text{ }\cVdash\exists x\mathbb{T}A & \text{\textit{modus attestans minor}} \\ 
R_{13} & \cVdash A\text{ }\&\cVdash B\Rightarrow \text{} \cVdash\lnot 
\mathbb{T}\lnot A\wedge \lnot \mathbb{T}\lnot B & \text{\textit{%
modus\ minor}}%
\end{array}%
\smallskip $

\noindent $%
\begin{array}{lll}
R_{Z}\text{ } & \hspace{-1.7pt}\cVvdash A(x/y)\text{ for all noemata }x\text{ }\Rightarrow 
\text{ }\cVvdash \forall xA(x/y) & \text{\textit{Z-regulation}}%
\end{array}%
\bigskip $

It is reminded that this list of regulations is not
complete, as librationism is not recursively axiomatizable and no such list
can be safeguarded as complete. Moreover, we have aimed at providing a
fairly comprehensive list instead of circumscribing a list of independent regulations. The \text{\textit{Z-regulation}} is named in analogy with the $ \omega $\text{\textit{-rule}} by using the last letter of the Latin alphabet.

Notice that on account of $ R_{1} $ conjoined with $ \pounds\aoverb{i}{M} - \pounds\aoverb{vi}{M} $ and the fact that all generalizations of the latter are presupposed as prescriptions, \pounds \ is super-classical. We take our regulation modus maximus above to be \textit{the} reasonable interpretation of modus ponens as classically intended, so that the novelty of our regulations does not as we see it constitute a weakening of but an extension of classical principles. However, the facial value of modus ponens only appeals to thesishod so we use the novel term to distinguish.

We verify regulation $R_{3}$. Suppose $\cVvdash A$ and $\cVdash A\rightarrow B$. We then have that $\fat \Koppa \vDash \mathbb{T}A$, $\fat \Koppa \vDash \lnot \mathbb{T}%
\lnot (A\rightarrow B)$ and $\fat \Koppa \vDash \lnot \mathbb{T}(A\rightarrow B)$.
That $\fat \Koppa \vDash \lnot \mathbb{T}\lnot (A\rightarrow B)$ means that $%
A\rightarrow B$ is unbounded under $\fat \Koppa $. That $\fat \Koppa \vDash \mathbb{T}A$
means that $A$ holds as from some ordinal below $\fat \Koppa $. As all ordinals below $\fat \Koppa$ are progression closed, $B$ is unbounded under $\fat \Koppa $, i.e. $%
\fat \Koppa \vDash \lnot \mathbb{T}\lnot B$. That $\fat \Koppa \vDash \lnot 
\mathbb{T}(A\rightarrow B)$ means that $A\wedge \lnot B$ is unbounded under $%
\fat \Koppa $. But so a fortiori also $\lnot B$ is unbounded under $\fat \Koppa $, i.e. 
$\fat \Koppa \vDash \lnot \mathbb{T}B$. So $\cVdash B$ and $\cVdash \lnot B$%
, i.e. $\cVdash B$. The other regulations are left at exercises, but \citep{Frode2012a} can be consulted.


\section{\euro, Countability, Order and Kind Choice}

Given the enumeration prescripts of section \ref{The semantics of pounds}, \euro \ is a bijection from $\mathbb{N}$ to a full universe as $\{x|x=x\}$; the non-extensionality of \pounds \ prohibits us from talking about \textit{the} full universe. Let $<$ be the usual order on $\mathbb{N}$:

\begin{definition}\label{deforder}~\\

i)\indent \hspace{3pt}$a\blt b\triangleq\exists x,y(x\in\mathbb{N}\wedge y\in\mathbb{N}\hspace{1pt}\wedge <x,a>\in\mbox{\euro}\hspace{2pt}\wedge <y,b>\in\mbox{\euro}\wedge x<y)$

~

ii)\indent $a\bleq b\triangleq a \blt b\vee a=b$

\end{definition}

The following exercise illustrates the paradoxicality of the power set operation and of sets given by the comprehension condition invoked in Cantorian arguments. Sections 7 and 8 of \cite{Frode2012a} should be compared.

\begin{exercise}
Show that if $\cVdash \{x|x=x\}\nsubset a$ then $C=\{x|x\in a\wedge \forall y(<x,y>\in\mbox{\euro}\rightarrow x\notin y)\}$, $f=\{<x,y>|<x,y>\in\mbox{\euro}\wedge y\subset a)\}$ and $\mathcal{P}(a)=\{x|x\subset a\}$ are paradoxical.
\end{exercise}

\begin{theorem}[The Kind Choice Theorem]~\\
If $a$ is a kind with only kind members, $\{x|\exists y(y\in a\wedge x\in y\wedge\forall z(z\in y\rightarrow x \bleq z))\}$ is a kind with precisely one member from each member of $a$.
\end{theorem} 

\noindent We leave the proof of the Kind Choice Theorem as an exercise.

\section{Curries and Set Theoretical Paadoxes}\label{curries}

Let $ c^{F}\triangleq \{x:x\in x\rightarrow F \}  $, for some sentence $F$. By $ A \aoverb {C}{M}$, $ c^{F}\in c^{F}\leftrightarrow\mathbb{T}(c^{F}\in c^{F} \rightarrow F) $ is a maxim. $\pounds\aoverb{8}{m}$ and sentence logic give us $\cVdash c^{F}\in c^{F}\rightarrow (c^{F}\in c^{F}\rightarrow F)$. But we have  $\cVvdash(c^{F}\in c^{F}\rightarrow (c^{F}\in c^{F}\rightarrow F))\rightarrow (c^{F}\in c^{F}\rightarrow F)$, so by modus
subiunctionis or modus maximus it follows that $\cVdash (c^{F}\in c^{F}\rightarrow F)$. By modus ascendens
we get $\cVdash \mathbb{T}(c^{F}\in c^{F}\rightarrow F)$, and so next $\cVdash c^{F}\in c^{F}$
follows from alethic comprehension by modus subiunctionis or modus maximus. So we have that
both $\cVdash c^{F}\in c^{F}$ and $\cVdash c^{F}\in c^{F}\rightarrow F$ for any arbitrary sentence 
$F$. Clearly, then, modus ponens as classically stated cannot hold for the roadstyle as otherwise \pounds \ would have been trivial. 

Now, $\cVdash c^{F}\in c^{F}$ being a thesis, it must either be a minor or a
maxim. If we have $\cVvdash c^{F}\in c^{F}$, we easily derive that $F$ is a maxim, i.e. $%
\cVvdash F$. If $c^{F}\in c^{F}$ is a minor it follows that 
$\lnot F$ is a thesis; for then also $\cVdash c^{F}\notin c^{F}$ so by
alethic comprehension and modus subiunctionis $\cVdash \lnot \mathbb{T}%
(c^{F}\in c^{F}\rightarrow F)$ and thence by modus scandens and modus descendens, $%
\cVdash c^{F}\in c^{F}\wedge \lnot F$, so that by tautologies and modus subiunctionis, $%
\cVdash \lnot F$. It follows, by parity of reasoning, that for any
sentence $F$, either $F$ is a maxim, $F$ is a
minor, or $\lnot F$ is a maxim. But in our contentual framework this is as we knew it to be, for we observed in the paenultimate paragraph of section \ref{The semantics of pounds} that \pounds \ is negjunction complete.

If we consider  $ c^{\bot}\triangleq \{x|x\in x\rightarrow \bot \}  $, where $ \bot $ is $ falsum $ or a contradiction, we realize that it behaves semantically as Russell's paradoxical set $ \{x|x \notin x\} $. Much more complicated paradoxicalities can be constructed when $ D $ in $ c^{D} $ iself is paradoxical.

\section{ The  Diagonal Lemma and Semantical Paradoxes}

\begin{theorem}[The Carnap-G\"{o}del Diagonal Lemma]\label{CarnapGodel}
If $A(x)$ is a formula with only $x$ as a noema, then there is a sentence $B$ such that $\cVvdash B\leftrightarrow A(\Godelnum{B})$.
\end{theorem}

\begin{proof}[Proof:]
We follow the related account of \cite{Smorynsky1977}, page 827. Given $A(x)$, define $A(Sub(x,x))$ (see page \pageref{Sub}) by replacing all occurrences of $x$ in $A$ with $Sub(x,x)$. Let $m=\Godelnum{A(Sub(x,x))}$. Consider $A(Sub(m,m))$. As $m=\Godelnum{A(Sub(x,x))}$, $\cVvdash A(Sub(m,m))\leftrightarrow A(Sub(\Godelnum{A(Sub(x,x))},m))$. As $\Godelnum{A(Sub(x,x))}$ is the code of a formula that only has $x$ as a noema, $\cVvdash A(Sub(\Godelnum{A(Sub(x,x))},m))\leftrightarrow A(\Godelnum{A(Sub(m,m))})$. So if we take $B$ as $A(Sub(m,m))$ we have that $\cVvdash B\leftrightarrow A(\Godelnum{B})$. 
\end{proof}

\vspace{5pt}
Recall Definition \ref{Concatenation} and consider the formula $\mathrm{T}x\Godelnum{\frown}\Godelnum{\rightarrow}\Godelnum{\frown}\Godelnum{D}$. By Theorem \ref{CarnapGodel} and Definition \ref{Concatenation} there is a sentence $C$ so that $\cVvdash C\leftrightarrow \mathrm{T}\Godelnum{C\rightarrow D}$. By using a variety of sentences $D$ we can generate a wide variety of paradoxical conditions, and other variations of the Carnap-G\"{o}del Diagonal Lemma ensure yet further variations as with the analogous set theoretic paradoxes of Section \ref{curries}.

\section{\mbox{\textit{Nominism}: Nominalism Released, Platonism Restrained}}
\label{sec:nominism}
As pointed out by Kripke it is common policy in logic to presuppose a \textit{generality-interpretation} of variables, or what we instead take as \textit{noemata}. Such a policy cannot be sustained in librationism. We need only consider e.g. $c^{v _{\underline{13}}={v_{\underline{96}}}}$ to see this. By the previous section, and as \pounds \ is negjunction complete, $ v _{\underline{13}}=v_{\underline{96}}$ is a thesis or $ v _{\underline{13}} \neq v_{\underline{96}}$ is a thesis. As identity is kind (see \citep{Frode2012a}) either $ v _{\underline{13}}=v_{\underline{96}}$ is a maxim of \pounds \ or $ v _{\underline{13}} \neq v_{\underline{96}}$ is a maxim of \pounds. Had one adopted the generality-interpretation for \pounds \ one would have to conclude that there is only one object according to \pounds \ lest \pounds \ be trivial. But a theory that is trivial or only postulates the existence of one object is not interesting and so we do not adopt the generality interpretation. Instead we adopt a \textit{nominality-interpretation} where what is usually taken as free variables are considered noemata and nomina, i.e. names. In our set up $ v _{\underline{13}} \neq v_{\underline{96}}$ is indeed a maxim, but we have no regulation or regula that allows generalizing universally.

We take \textit{nominism} to be the view that all mathematical objects have a name while it nonetheless upholds the platonist view that mathematical objects are abstract. Nominism is supported by librationism by the latters avoidance of Cantor's conclusion that there are uncountable infinities and insistence instead that there are only denumerably many obects; for the reader's comvenience we repeat that it is not claimed anywhere that Cantor's arguments are invalid (cfr. \citep{Frode2012a} and \citep{Frode2011} for more on this). As there, according to \pounds, are only a denumerable infinity of objects, we have enough names to name all mathematical objects and we have invoked a nominality policy with that as objective in our semantics. 

Notice that the term ``platonism'' in the context of set theory is sometimes taken to stand for the \textit{Cantorian} view that the endless hierarchy of alephs exists in a non-relative sense. Here we instead abide by what we take as a more plausible usage of the term in the philosophy of mathematics where it denotes the view that takes mathematical objects to be abstract objects. The nominism we propound of course rejects the Cantorian version of platonism.

\section{Contradiction, Contravalence and Complementarity}
\label{Negjunction}

In this section we develop new ideas which pertain to how the connectives should be understood in the librationist framework. 


As we have seen in section \ref{The semantics of pounds}, the semantics of \pounds \ comes about by elaborations, and a librationist twist, upon the semi inductive type of approach employed by \citep{Herzberger1982} to analyse the Liar's paradox in a way which in some important respects improved upon that of \citep{Kripke75}. However, there is not a simple duality between the librationist semantics and an envisioned Herzbergerian style approach as \pounds \ has noemata and no free variables. $ v_{\underline{13}} \neq v_{\underline{16}} $ may e.g. be a maxim of \pounds \ but $ v_{\underline{13}} = v_{\underline{16}} $ is \textit{not} unbounded in all Herzbergerian semantic processes. One important aspect of the librationist elaboration upon the Herzberger style approach is that we focus upon \textit{one} semantic process in order to gain categoricity; this we dub the \textit{nominist turn}, and it is with the nominist turn accompanied by the librationist twist that \pounds \ becomes a \textit{contentual system}.

Induced principles for truth  and abstraction are
such as to e.g. make $r\in r$ a thesis of \pounds \ and also the negjunction $r\notin r$ a thesis of \pounds \ when $r$ is Russell's set $\{
x:x\notin x \}$ of all and only those sets which are not members of
themselves. However, $r\in r\wedge r\notin r$ is not a thesis
of \pounds; notice that if it had been this would have contradicted
that \pounds \ is a sober extension of classical logic. Pay also heed here to
the fact that the induced regulations of \pounds \ are all novel
and in particular that modus ponens as classically stated is not generally truth preserving in \pounds. 

Librationism indeed comes very close to being a so-called
paraconsistent system. However, I shall on the basis of its semantics argue
that \pounds \ is neither paraconsistent nor inconsistent nor
contradictory. This will involve concepts and a terminological policy from
a crestal (“meta”) level. 

A sentence's \textit{valency} is the set of von Neumann
ordinals of \textbf{L} in the semantical process where it holds. The \textit{valor} of a sentence is
the least upper bound of its valency, and thus 
the union of its valency as we presuppose von Neumann ordinals in the semantic set up.  A sentence has
at least one of two \textit{values}, viz. it can be $ true$  or it can be $ false$. A sentence
is true iff its valor is the closure ordinal \fat \Koppa \ and a formula is false iff its negjunction is true. 

We let ${\mathbb{V}}( A )$   be the valency of $ A $ and   $\mathcal{V}( A )$ be the valor of $ A $. If $\mathcal{V}( A )\prec\fat \Koppa$, $ A $ is false and not true
and we say from our external perspective that $ A $ is \textit{pseudic};  if $\mathcal{V}( \lnot A )\prec\fat \Koppa$, $ A $ is true and not false and we say that $ A $ is \textit{veridic}. It turns out that a sentence $ A $ is veridic (pseudic)
iff it is a maxim that $\mathbb{T} A $ ($\mathbb{T}\lnot A$), and the latter holds iff $ A $ ($\lnot A$) is a maxim. There are no special
challenges with accounting for paradoxical conditions that involve the notions
of being pseudical or veridical, for \pounds \ has no access to the external perspective on \pounds \ needed to talk about veridicity and pseudicity. 

Two sentences $ A $ and $ B $ are \textit{parivalent} iff the
valency of $ A $, $\mathbb{V} (A)$, is identical with $\mathbb{V} (B)$; two sentences are \textit{altervalent} iff they are not parivalent; two sentences $ A $ and
$ B $ are \textit{contravalent} iff $\mathbb{V} (A)$ is $\fat \Koppa\backslash \mathbb{V} (B)$; the \textit{ambovalence} of
$ A $ and $ B $ is the intersection $\mathbb{V} (A) \cap \mathbb{V} (B)$; the \textit{velvalence} of $ A $ and $ B $ is
the union $\mathbb{V} (A)\cup \mathbb{V} (B)$; the \textit{subvalence} of $ A $ under $ B $ is $\mathbb{V} (\lnot
A)\cup \mathbb{V} (B)$ and the \textit{homovalence} of $ A $ and $ B $ is
$(\mathbb{V} (A) \cap  \mathbb{V} (B)) \cup (\mathbb{V} (\lnot A) \cap \mathbb{V} (\lnot B)) $.

The connectives do not work truth-functionally in librationism, but they work \textit{valency-functionally} and by following the classical interdefinability connections as in any Boolean algebra. The valency of the negjunction of $ A $, $\mathbb{V} (\lnot A)$, is the
contravalence of $ A $; the valency of the \textit{adjunction} («conjunction») of $ A $ and $ B $,
$\mathbb{V} (A\wedge B)$, is the ambovalence of $ A$ and $ B $; the valency of the
\textit{veljunction} («disjunction») of $A$ and $ B $,  $\mathbb{V} (A\vee B)$, is the velvalence of
$ A $ and $ B $; the valency of the \textit{subjunction} (“material implication”) $A\rightarrow B$,
$\mathbb{V} (A\rightarrow B)$, is the subvalence of $ A $ under $ B $ and the valency of the
\textit{equijunction} (“material equivalence”) $A\leftrightarrow B$, $\mathbb{V} (A\leftrightarrow B)$, is the
homovalence of $ A $ and $ B $. 

We use sligthly non-standard names for the connectives in part to forestall irrelevant objections which appeal to something like the one and only true meaning of them in ordinary language. One should appreciate that the connectives do behave truth functionally for maxims, and so in this essential respect they do have their classical meaning in ordinary discourse.

Two sentences $ A $ and $ B $ are \textit{incompatible} iff their
ambovalence is empty and, consequently, $A\wedge B$ fails to be a thesis.
Contravalent sentences are thence also incompatible. All sentences are
contravalent with their negjunction; moreover, if $ A $ and $ B $ are contravalent then
$ A$  and $\lnot B$ are parivalent.  

What does ``What does it say?''\hspace{-5pt} say? A variety of things, to be sure; but we lay down the convention that a sentence \textit{dictates} its valor, and that a sentence \textit{presents} its valency. Given this convention, we take contravalent sentences to be \textit{contrapresentive} and parivalent sentences to be \textit{paripresentive}. We on occasion say that the sentence's valency is the \textit{way} the sentence dictates its valor. 

Two sentences are \textit{paridictive} iff they dictate the same valor, and otherwise they are \textit{alterdictive}.
Two sentences are contradictive iff they are contravalent and alterdictive. A contradiction is the adjunction of two contradictive sentenes. In
consequence precisely one of two contradictive sentences is true, so it can by
our light never be correct, \textit{pace} certain paraconsistentists, to maintain
contradictions. 

Two sentence are \textit{complementary} iff they are
contravalent and paridictive. Two complementary sentences are thence both true, and they dictate the same (viz.
\fat \Koppa) in \textit{opposite} ways. In particular, for Russell's set $r=\{x|x\notin x\}$, $r\in r$ and $r\notin r$ are
complementary. Thence $r\in r$ and $r\notin r$ dictate the same in opposite ways and so do not by our standards contradict each
other. 

We cannot strictly speaking say ``sentences
$r\in r$ and $r\notin r$ are both true'' in one mouthful as $ r\in r $ and $ r\notin r $ are incompatible. Librationism is committed to a peculiar discretist or
sententialist, though nonetheless holistic, point of view.

We regard
a theory as inconsistent iff it has theses of the form $A\wedge\lnot A$, and as contradictory iff it has contradictions as theses; consequently, we hold that a theory is inconsistent iff it is contradictory. Given the
foregoing it is our considered opinion that \pounds \ is consistent and not
contradictory. As a consequence librationism should not be considered to be
a dialetheist point of view as \textit{dialetheism} is canonically characterized as a view
which takes some contradictions to be true. Instead we take librationism to
offer a \textit{bialethic} point of view. 

We take our considerations in this section to meet a challenge that remains one for those who as some dialetheists believe that a sentence and its contradiction can both be true. Such paraconsistentists owe others an explanation as to what they in dialetheic cases think it is a true sentence $p$ says which its true contradictory $\lnot p$ contradicts.
\section{Coherency, Incoherency and Paracoherency}

We take a formula $A$ of a theory $T$ to be a \textit{nonthesis} of $T$ iff $A$ is not a thesis of $T$, and understand the theory $T$ to be an ordered pair $\fat <T_{\vDash},T_{\nvDash}\fat >$ where $T_{\vDash}$ is the set of theses of $T$ and $T_{\nvDash}$ is the set of nontheses of $T$. By the observation towards the end of Section \ref{The semantics of pounds} that in \pounds \ we have $\cVvdash A$ iff ${\cancel \cvdash} \lnot A$, both thesishod of \pounds \ and nonthesishood of \pounds \ are instrumental in the statements of the posits of \pounds \ (Section \ref{prescriptions}) and in the statements of its regulations (Section \ref{Regulations}). We also understand classical theories as $PA$ and $ZF$ in such terms as this in that we take \textit{modus maximus} as the proper understanding of \textit{modus ponens} and take all axioms or axiom schemas $A$ of the theory to be asserted while the negation taken as a nonthesis of the theory. A theory $\fat <T_{\vDash},T_{\nvDash}\fat >$ is \textit{coherent} iff $T_{\vDash}\bigcap T_{\nvDash}$ is empty, and a theory is \textit{incoherent} iff it is not coherent.

A sentence $A$ is taken as a \textit{contrapresentive} thesis of a theory iff $A$ is both a thesis and an anti thesis of the theory. Let us agree that a theory is contrapresentive iff it has contrapresentive theses. A theory is trivial iff all sentences of its language are theses. Trivial systems and inconsistent theories with simplification or adjunction elimination are contrapresentive. \pounds \ is contrapresentive, but neither trivial nor inconsistent. Contrapresentationism is the view that a contrapresentive theory, such as \pounds, is true.

We say that two formulas $A$ and $B$ of a theory $T$ are \textit{connected} iff $A$ and $B$ are theses of $T$ only if also the conjunction $A\wedge B$ is a thesis of $T$. Two sentences disconnect with each other iff they do not connect with each other. A sentence is disconnected iff it disconnects with some sentence. A set $b$ is disconnected iff for some set $a$ the sentence $a \in b$ is 
disconnected. Paradoxical theses of \pounds \ are disconnected theses of \pounds, and vice versa. It is straightforward that if $A$ is not a thesis then A connects with all sentences, and further that all maxims connect with all sentences. All sentences are self-connected and the relation $connects \ with$ is also symmetric, but not transtive. A theory is disconnected iff it has disconnected theses. A topic is disconnected iff a true theory about it is disconnected and \textit{disconnectionism} is the view that there are sound disconnected theories 

Some paraconsistent logics, such as the ones following the approach by Jaskowski, are \textit{non-adjunctive}. But such logics do not in and of themselves have disconnected theses, though extensions of such logics with suitable comprehension principles or semantic principles may be disconnected if not trivial.

Notice also that \pounds \ is distinct from paraconsistent systems in that the latter, including Jaskowski's system, do not have \textit{ex \ falso \ quodlibet}, ($A \wedge \lnot A \rightarrow B$), as a theorem whereas \pounds \ has it as a maxim. We emphasize this obvious corollary for the reason that the failure of \textit{ex falso quodlibet} is very arguably a defining feature of paraconsistent approaches.

We correctly hold it against someone if she utters $A$ and thence $not \ A$ when she is meant to elucidate a connected topic. However, we take our discussion to have revealed that paradoxes are essentially disconnected. I therefore contend that \pounds \ should be regarded as fulfilling a very important adequacy requirement by being a disconnected theory about the disconnected and quasi incoherent and absurd topic of paradoxicality. 


The librationist theory \pounds \ may, as far as its dealing with
paradoxical phenomena is concerned, be thought of as one accompanied with many
perspectives which we shift between in our reasoning in those contexts. One may employ the term \textit{parasistency} for this idea that \pounds \ lets us \textit{stand beyond}, so to speak, and shift between perspectives. We have
an external overarching crestal (meta) perspective which helps us illuminate many issues, as
e.g. illustrated here. Besides, for each true sentence  its valency may be
thought of as associated with a partial perspective we may take upon
mathematical and semantical reality. There are, then, oscillating switches, as it were,
between different perspectives we may have upon paradoxical sentences, and this
justifies our adoption of librationism's neologist name. Moreover, such
oscillating shifts between perspectives correspond well with our contemplative
experiences in connection with the paradoxes, and so librationism on this count
appropriately fulfills an important desideratum, arguably an adequacy
condition, by incorporating them. 

It is of course important in all of this that the right balance is struck between various desiderata. As it is, \pounds \ seems well suited to strike precisely such a balance as it soberly extends classical logic and even soberly interprets classical mathematics in as far as the latter is consistent. Moreover, theses are only disconnected theses of \pounds \ if they are minor and hence paradoxical theses of \pounds. Our isolated notions of complementarity and valency functionality appropriately alleviate the loss of intuitiveness brought upon us by the paradoxical phenomenon of disconnectedness.

\section{Manifestation Points with Parameters} \label{manifestation point}

We have earlier indicated that we follow the standard practice of writing $A(x,y)$ to indicate that the noemata $x$ and $y$ are present in the formula $A$. We extend this practice so that we e.g. write $c(x,\vv{z})$ to indicate that the noema $x$ and the noemata from the vector $\vv{z}$ occur as parameters in the term $c$; as we now work metalinguistically it makes sense to think as if we invoke noemata as parameters.

We rehearse the definition of manifestation point of section 6 of \cite{Frode2012a} for convenience and to show that we can make use of instances with parameters.

\begin{theorem} 
If $A(x,y,z)$ is a formula with the noemata indicated we
can find a term $h(z)$ such that $\cVvdash\forall w,z(w\in h(z)\leftrightarrow\mathbf{TT}%
A(w,h(z),z))$.
\end{theorem}

\begin{proof}[Proof:]
Let $d=\{\langle x,g \rangle|A(x,\{u|\langle u,g \rangle\in g\},z)\}$ and $h(z)=\{y|\langle y,d \rangle\in
d\}$. If we spell out, we have: $h(z)=\{y|\langle y,\{\langle x,g \rangle|A(x,\{u|\langle u,g \rangle\in
g\},z)\}\rangle\in\{\langle x,g \rangle|A(x,\{u|\langle u,g \rangle\in g\},z)\}\}$. By $\mathrlap{\mathbf{C}}A_{\mathrm{M}} $
we have that $\cVvdash\forall w(w\in h(z)\leftrightarrow\mathbb{T}%
\langle w,\{\langle x,g \rangle|A(x,\{u|\langle u,g \rangle\in g\},z)\}\rangle\in\{\langle x,g \rangle|A(x,\{u|\langle u,g \rangle\in g\},z)\})$,
and so by $\mathrlap{\mathbf{C}}A_{\mathrm{M}} $ again it follows that $\cVvdash\forall w(w\in
h(z)\leftrightarrow\mathbf{TT}A(w,\{u|\langle u,\{\langle x,g \rangle|A(x,\{u|\langle u,g \rangle\in g\},z)\}\rangle\in
\{\langle x,g \rangle|A(x,\{u|\langle u,g \rangle\in g\},z)\}\},z)$. But this is $\cVvdash\forall w(w\in
h(z)\leftrightarrow\mathbf{TT}A(w,h(z),z))$, as wanted.  As we only invoked generalizable prescriptions also $\cVvdash\forall w,z(w\in
h(z)\leftrightarrow\mathbf{TT}A(w,h(z),z))$. Obviously, we can have more than one parameter or a vector of parameters in $A$.
\end{proof}

\section{Defying an Orthodox Identity}

Theorem 5 of \citep{Frode2012a} implies that \pounds \ is highly non-extensional in that for any kind $a$ there are infinitely many mutually distinct kinds coextensional with $a$. The main result of this section has the same as consequence for non-empty kinds, but is a more general result and based upon different considerations.

\begin{lemma} $\cVvdash a=\{x|x \in a\}$ only if a is kind. \end{lemma} 

\begin{proof}[Proof:] Suppose $\cVvdash a=\{x|x \in a\}$. By substitution of identicals and alethic comprehension $\mathrlap{\mathbf{C}}A_{\mathrm{M}} $, $\cVvdash\forall x\mathbb{T}(\mathbb{T}x \in a\leftrightarrow x \in a)$, so $\cVvdash\forall x\mathbb{T}(\mathbb{T}x \in a\rightarrow x \in a)$. By $\pounds^5_M$  $\cVvdash\forall x(\mathbb{T}(\mathbb{T}x \in a\rightarrow x \in a)\rightarrow (\mathbb{T}x \in a\mathbf{ \vee T}x \notin a)) $, so by classical logic $\cVvdash\forall x(\mathbb{T}x \in a\mathbf{ \vee T}x \notin a)$ which by the terminology of \citep{Frode2012a} means that $a$ is kind. \end{proof}

\begin{lemma} \ \ \phantomsection\label{first} If $\alpha \vDash a=b$ and $\alpha \succ 0$   then $\alpha \vDash \mathbb{T}a=b$ \end{lemma}

\begin{proof}[Proof:] As for Lemma 1 of \cite{Frode2012a}, and using the librationist definition of identity and the fact that $\alpha\vDash a \in \{x|x=a\}$. \end{proof}

\begin{lemma} \label{second} $\alpha\vDash a=b$ and $\beta \prec \alpha$ only if $\beta \vDash a=b$ \end{lemma}

\begin{proof}[Proof:] Suppose $\alpha \vDash a=b$ and $\beta\fat\prec\alpha$ for some ordinal $\beta$. It follows that $\alpha\fat\succ 0$, so that by Lemma \ref{first} then $\alpha \vDash \mathbb{T}a=b$. It follows that for some $\delta \fat\prec \alpha$, for all $\alpha\fat\succ\gamma \fat\succeq\delta$, $\gamma \vDash a=b$; in particular $\delta \vDash a=b$. By the wellfoundedness of ordinals one can iterate the process to any ordinal $\beta$ below $\alpha$.  \end{proof}

\begin{lemma} \label{third} If $\cVvdash a=\{x|x \in a\}$ and $\beta \prec \koppa$ then $\beta \vDash \forall x (x \notin a)$\end{lemma}

\begin{proof}[Proof:] 
Suppose to the contrary that for some $b$ and $\beta$: $\cVvdash a=\{x|x \in a\}$, $\beta \fat\prec \koppa$ and $\beta \vDash b \in a$. By Lemma \ref{second} and the theory of identity, $\beta \vDash b \in \{x|x \in  a\}$; by alethic comprehension $\mathrlap{\mathbf{C}}A_{\mathrm{M}} $, $\beta \vDash \mathbb{T} b \in a$; so for some $\delta \fat\prec \beta$, $\delta \vDash b \in a$. But this could be iterated so that we got $0 \vDash b \in a$, so that $1 \vDash a \in \{x|b\in x\}$. But by Lemma \ref{second} again and our theory of identity, then also $1 \vDash \{x|x\in a\} \in \{x|b\in x\}$, so that  $0 \vDash b \in \{x|x\in a\}$; but the latter statement is impossible on account of principle $P(1)$ of our semantics as $0$ is the smallest ordinal and $\{x|x\in a\}$ is of caliber zero (see Definition \ref{caliberdef}).\end{proof}

\begin{lemma} \label{fourth} If for all $\beta \fat\prec \koppa$, $\beta \vDash \forall x (x \notin a)$, then $\cVvdash a=\{x|x \in a\}$ \end{lemma}

\begin{proof}[Proof:] Suppose to the contrary.  By the logic of identity for all $\beta \fat\prec \koppa$, $\beta \vDash \forall x (x \notin a)$ and $\cVvdash a\neq\{x|x \in a\}$. If for all $\beta \fat\prec \koppa$, $\beta \vDash \forall x (x \notin a)$ also for all $\beta \fat\prec \koppa$, $\beta \vDash \forall x (x \notin \{x|x \in a\})$; this we can infer without invoking the identity between $a$ and $\{x|x \in a\}$. $0\vDash a=\{x|x \in a\}$, as $0\vDash b=c$ for all terms $b$ and $c$. As $\cVvdash a\neq\{x|x \in a\}$ there must be a first ordinal $\gamma$ such that $\gamma\vDash a\neq \{x|x \in a\}$. Let $\delta$ be any ordinal smaller than $\gamma$. From the information we have $\delta\vDash b\in a \leftrightarrow b\in \{x|x\in a\}$ and $\delta\vDash a\in b \leftrightarrow \{x|x\in a\}\in b$ for all terms $b$; moreover, such biconditionals also hold within any number of applications of the defined truth operator $\mathbb{T}$. Consequently, there cannot be anything below $\gamma$ which justifies a sudden change to non-identity between $a$ and $\{x|x\in a\}$. \end{proof} 

\begin{lemma} \label{fifth}If for all $\beta \fat\prec \koppa$, $\beta \vDash \forall x (x \notin a)$, then $\cVvdash a= \emptyset$ \end{lemma}

\begin{proof}[Proof:] Here $\emptyset = \{x|x\neq x\}$. The argument is as for Lemma \ref{fourth}. \end{proof}

\begin{theorem} $\cVvdash a=\{x|x \in a\}$ iff $\cVvdash a=\emptyset$ \label{surprise on identity} \end{theorem}

\begin{proof}[Proof:] The argument is by combining the preceding lemmas.\end{proof}

\medskip

The non-extensionality results of \citep{Frode2012a} and preceding results in earlier literature referred to there are surprising in the sense that they force us to revise the prejudice of extensionality in type free contexts by means of explicit conditions that convince us. Theorem \ref{surprise on identity} is different in that it surprises us in its statement of the very wide prevalence of non-extensionality in cases where one should not expect it at all. One may hold against this that some of these features are consequences of a perhaps arbitrary decision of having all terms of the form $\{x|A(x)\}$ empty at the ordinal $0$; but other choices would have had their own blend of arbritrary seeming consequences, and the author thinks that there are so good aesthetical reasons concerned with uniformity to support the semantical set up presupposed that it is not arbritrary.

\section{{\large Capturing Collection, Replacement, Specification and Choice}} \label{Aus}

We introduce \textit{ordinary} capture by the following definition, where $\alpha(x,y)$ is any first order condition on $x$ and $y$ as standardly conceived in classical set theory. 

\begin{definition}[Ordinary capture]~\\

\medskip

$\forall v\exists w\forall x(x\in w \leftrightarrow\exists y(y\in v \wedge \alpha(y,x) \wedge (\forall z)(\alpha(y,z)\rightarrow x=z)))$
\end{definition}

Let $W$ be Zermelo set theory $Z$ minus the axiom schema of specification.

\begin{theorem}
$ZF = W + ordinary \ capture$
\end{theorem}
\begin{proof}[Proof:]
The condition $\beta(y,x)=\alpha(y,x) \wedge (\forall z)(\alpha(y,z)\rightarrow x=z)$ is functional, so capture follows from replacement. If $\alpha(y,x)$ is already functional the clause requiring only one set captured is redundant, so capture eintails replacement. Specification follows from capture by letting $\alpha(y,x)=\beta(y)\wedge y=x$. 
\end{proof}

We introduce \textit{librationist} capture by the following definition, where $\alpha$ is as above and the order relation $\bleq$ as in Definition \ref{deforder}:

\begin{definition}[Librationist capture]~\\

\medskip

If $a$ is a term, so is $\{x|\exists y(y\in a \wedge \alpha(y,x) \wedge (\forall z)(\alpha(y,z)\rightarrow x\bleq z)))\}$.
\end{definition}

\begin{proposition}[Librationist capture of choice]\label{capturechoice}~\\

\noindent If $c$ is a kind with only nonempty kinds whose intersections are pairwise empty for members not coextensional, then $\{x|\exists y(y\in c \wedge x\in y \wedge (\forall z)(z\in y\rightarrow x\bleq z)))\}$ is a kind  that, modulo extensionality, contains precisely one kind from each kind of $c$. So librationist capture entails full choice in those contexts it is operative.
\end{proposition}

We will see in section \ref{definable echelon} that librationist capture in some contexts where extensionality fails is stronger than ordinary capture also in that the former provides collection and the latter not. Our considerations in this section are of a general nature, but the relationships pointed out between capture, replacement, specification and choice carry over to the more intricate constructions carried out in the next section and in section \ref{definable echelon}.

Our use of capture is not merely motivated by its pleasing aesthetical qualities such as that it is a biconditional or that we avoid special restrictions on the invoked first order condition, but also by the fact that it is a more manageable closure principle which is more useful for our definitional purposes below.

\section{Varieties of Conditions}

We explore some uniform maximal closure conditions for $\mathbb{H}$ beyond the Jensen rudimentary functions amounting to bounded separation discussed in \citep{Frode2012a}. 

Recall definition \ref{freevariables}. We make an analogous object language definition:

\begin{definition}

\begin{align*}
\Godelnum{\textcjheb{n}}(\Godelnum{v_{i}}) & \triangleq  \{\Godelnum{v_{i}}\} \\
\Godelnum{\textcjheb{n}}(\Godelnum{\mathrm{T}})&\triangleq \{\Godelnum{\mathrm{T}}\} \\
\Godelnum{\textcjheb{n}}(\Godelnum{\mbox{\euro}})&\triangleq \{\Godelnum{\mbox{\euro}}\}\\
\Godelnum{\textcjheb{n}}(\Godelnum{ba}) &\triangleq \Godelnum{\textcjheb{n}} (a)\cup \Godelnum{\textcjheb{n}} (b) \\
\Godelnum{\textcjheb{n}}(\hspace{1pt} \Godelnum{\hat{}\hspace{2pt}v_{\underline{i}}A}) &\triangleq \Godelnum{\textcjheb{n}}(A)\setminus\{{\Godelnum{v_{\underline{i}}}}\} \\
\Godelnum{\textcjheb{n}}(\Godelnum{\forall v_{\underline{i}}A}) &\triangleq \Godelnum{\textcjheb{n}}(A)\setminus\{{\Godelnum{v_{\underline{i}}}}\} \\
\Godelnum{\textcjheb{n}}(\Godelnum{\downarrow \hspace{-2pt}AB}) &\triangleq \Godelnum{\textcjheb{n}} (A)\cup \Godelnum{\textcjheb{n}} (B) \\
\Godelnum{\textcjheb{n}}(\Godelnum{\downarrow \hspace{-2pt}ab}) &\triangleq \Godelnum{\textcjheb{n}} (a)\cup \Godelnum{\textcjheb{n}} (b) \\
\end{align*}

\end{definition}


We define \textit{conditions} upon noemata.

\begin{definition}\label{Cnd}$Cnd(x,y)\triangleq x$ is the code of a formula that does not contain \, $\hat{}$\hspace{1pt}, $\mathrm{T}$ or \euro \ and contains the joint only as connective and not as a juncture and $\forall z(z\in \Godelnum{\textcjheb{n}}x \rightarrow z\in y) $. 
\end{definition}

We will e.g. have $Cnd(\ulcorner\alpha\urcorner,\{\ulcorner v_{\underline{0}}\urcorner\})$ just if $\alpha$ is a first-order condition upon $v_{\underline{0}}$ precisely as in classical set theory without the identity sign as primitive.

\smallskip For $X$ a possibly more restricted set of (codes of) formulas we define:

\begin{definition}
$\textit{X-Cnd(x,y)}\triangleq Cnd(x,y)$ and $x\in X$. 
\end{definition}




In the following we usually presuppose the full set of first order conditions. 

Recall our substitution function {\scriptsize SUB} of Section \ref{substitution}.


\begin{definition}
z \ is \ a \ Capture \ without \ parameters \ of \ a via \ x

\medskip

\noindent$Cpt(z,a,x) \triangleq Cnd(x,\{\ulcorner v_{\underline{0}}\urcorner,\ulcorner v_{\underline{1}}\urcorner\})\wedge z=\{u|\exists v(v \in a \wedge \mathrm{T}\mbox{{\scriptsize SUB}}(\mbox{{\scriptsize SUB}}(x,v),u)\wedge \\ \indent \forall w(\mathrm{T}\mbox{{\scriptsize SUB}}(\mbox{{\scriptsize SUB}}(x,v),w)\rightarrow u=w)  \}$
\end{definition}




\begin{definition} z \ is \ a \ Capture \ without \ parameters \ of \ $a$

\medskip
$Cpr(z,a)\triangleq \exists xCpt(z,a,x)$
\end{definition}





\section{Domination}

For some gain in effectivity, we state the domination requirements in terms of capture.

\begin{definition}[Domination requirement with parameter of parameters]~\\

$D(x,y,a,z) \triangleq\exists t(Cnd(t,z)\wedge \Godelnum{v_{\underline{0}}}\in\Godelnum{\textcjheb{n}}(t)\wedge \Godelnum{v_{\underline{1}}}\in \Godelnum{\textcjheb{n}}(t)\wedge\Godelnum{v_{\underline{2}}}\in\Godelnum{\textcjheb{n}}(t)\wedge\\ \indent  s=\mbox{\scriptsize SUB}(t\Godelnum{\frown}\Godelnum{\wedge}\Godelnum{\frown}\Godelnum{v_{\underline{1}}=v_{\underline{2}}},y)\wedge\\ \indent x=\{u|\exists v(v \in a \wedge \mathrm{T}\mbox{{\scriptsize SUB}}(\mbox{{\scriptsize SUB}}(s,v),u)\wedge\forall w(\mathrm{T}\mbox{{\scriptsize SUB}}(\mbox{{\scriptsize SUB}}(s,v),w)\rightarrow u=w)  \})$


\end{definition}

\begin{definition}[Domination requirement without parameters]~\\

$D^{-}(x,y,a) \triangleq\exists t(Cnd(t,\{\Godelnum{v_{\underline{0}}},\Godelnum{v_{\underline{1}}},\Godelnum{v_{\underline{2}}}\})\wedge \Godelnum{v_{\underline{0}}}\in\Godelnum{\textcjheb{n}}(t)\wedge\Godelnum{v_{\underline{1}}}\in\Godelnum{\textcjheb{n}}(t) \\ \indent\wedge\Godelnum{v_{\underline{2}}}\in\Godelnum{\textcjheb{n}}(t)\wedge s=\mbox{\scriptsize SUB}(t\Godelnum{\frown}\Godelnum{\wedge}\Godelnum{\frown}\Godelnum{v_{\underline{1}}=v_{\underline{2}}},y)\wedge\\ \indent x=\{u|\exists v(v \in a \wedge \mathrm{T}\mbox{{\scriptsize SUB}}(\mbox{{\scriptsize SUB}}(s,v),u)\wedge\forall w(\mathrm{T}\mbox{{\scriptsize SUB}}(\mbox{{\scriptsize SUB}}(s,v),w)\rightarrow u=w)  \})$
  
\end{definition}



\noindent Domination withouth parameters can now be defined by manifestation point with one parameter as described in Section \ref{manifestation point} above:

\begin{definition}
$\cVvdash\forall x(x \in \mathcal{D}^{-}(h) \leftrightarrow \mathbf{TT}D^{-}(x,\mathcal{D}^{-}(h),h))$
\end{definition}

\begin{lemma}
$\cVvdash\forall x(KIND(\mathcal{D}^{-}(x)))$
\end{lemma}

\begin{proof}[Proof:]
As \pounds \ soberly extends the theory of identity and Peano arithmetic.
\end{proof}

\begin{corollary}
$\cVvdash\forall x(x \in \mathcal{D}^{-}(h) \leftrightarrow D^{-}(x,\mathcal{D}^{-}(h),h))$
\end{corollary}

\begin{definition}[Ordinary specification, short hand]~\\

$\overset{o}{\langlebar} x\in a|\alpha(x)\overset{o}{\ranglebar}\triangleq\{x|\exists y(y\in a\wedge\mathrm{T}\Godelnum{\alpha(\dot{y})\wedge \dot{y}=\dot{x}}\wedge\forall z(\mathrm{T}\Godelnum{\alpha(\dot{y})\wedge \dot{y}=\dot{z}}\rightarrow z=x)\}$
\end{definition}

\begin{lemma}[The Impredicativity Lemma]~\\ \label{impredicativity}

\noindent If $z=\{u|u \in \mathcal{D}^{-}(h) \wedge \alpha(u) \}$, $Cnd(\Godelnum{\alpha},\{\Godelnum{u}\})$ and kind $h$ then a kind $w$ coextensional with $\bigcup z$ is a member of $\mathcal{D}^{-}(h)$.
\end{lemma}
\begin{proof}[Proof:]
$\mathcal{D}^{-}(h)$ is impredicative as $ \overset{o}{\langlebar}x\in h|\exists u(u\in\mathcal{D}^{-}(h)\wedge\alpha(u)\wedge x\in u)\overset{o}{\ranglebar}\in\mathcal{D}^{-}(h)$. In desired kind contexts the latter is coextensional with $\bigcup\{u|u \in \mathcal{D}^{-}(h) \wedge \alpha (u)\}$.
\end{proof}


\section{The Skolem Cannon}

As in \cite{Frode2012a} we take a sort to be good iff it is hereditarily kind and in $\mathbb{H}$.

\begin{cannon}[The Skolem Cannon]\label{Cannon}~\\

If $c$ is good and $Cpr(z,c)$ (with possible parameters from $\mathbb{H}$), then so is z.\\ \indent More precisely, $\cVvdash c\in\mathbb{H}\land Cpr(z,c)\to z\in\mathbb{H}$.
\end{cannon}

The indebtedness to Skolem's idea of replacing Zermelo's imprecise notion of \textit{definite Aussage} with the notion now known as \textit{first order condition} is manifest. Nevertheless, for stylistic purposes we use  \textit{the Cannon} as a metonym  for the title of this section. Notice that the Cannon is not an original posit, and it amounts to a leap of faith that it is a theorem, as it were, or a presupposition on a par with the credo in or lore of the ZF tradition that replacement or collection with first order conditions does not lead to inconsistency. Given G\"{o}del's second incompleteness theorem, we cannot expect to prove that the Cannon holds in the librationist semantic set up unless with means that surpasses the resources used in that set up. We state the two critical adjoint theorems conditional upon the Cannon: 

\begin{theorem}
$\cVvdash(\forall x)(x \in \mathbb{H}\rightarrow\mathcal{D}^{-}(x) \in \mathbb{H})$
\end{theorem}

\begin{theorem}
$\cVvdash(\forall x)(\forall y)(x \in \mathbb{H} \wedge Cpr(y,x)\rightarrow y \in \mathbb{H})$
\end{theorem}




\section{Definable Real Numbers}

\medskip

Let $\mathbb{Q}$ be the set of rational numbers and $<$ their usual order. We use a formula with precisely the noemata $x$ and $y$ to appropriately capture Dedekind cuts when defining the set $\mathbb{D}$ of definable real numbers:

\begin{definition}[The Definable Real Numbers]~\\

$\mathbb{D} \triangleq\overset{o}{\langlebar} x \in \mathcal{D}^{-}(\mathbb{Q})|\exists u(u \in x)\wedge\exists u(u \in \mathbb{Q}\wedge u \notin x)\wedge \forall u(u\in x \rightarrow \\ \indent\forall v(v\in \mathbb{Q} \wedge v<u \rightarrow v\in x)) \wedge \forall u(u\in x \rightarrow \exists v(v\in \mathbb{Q} \wedge u<v \wedge v\in x))\overset{o}{\ranglebar}$

\end{definition}

\begin{theorem}
If $z$ is a definable set of real numbers from $\mathbb{D}$ with an upper bound in $\mathbb{D}$ then $z$ has a least upper bound in $\mathbb{D}$.
\end{theorem}

\begin{proof}[Proof:]
The hint is to invoke the Impredicativity Lemma \ref{impredicativity}.
\end{proof}

\medskip

\section{The Definable Echelon}\label{definable echelon}

\begin{definition}[Conditions restricted to $u$]\label{Cnd}~\\

$Cn(x,u)\triangleq x$ is the code of a formula that does not contain \, $\hat{}$\hspace{2pt}, $\mathrm{T}$ or \euro \ and contains \indent the joint only as connective and not as a juncture and $\forall z(z\in \Godelnum{\textcjheb{n}}x \rightarrow z\in u) $ and all \indent the quantifiers of the formula x is a code of are restricted to $u$. 
\end{definition}



\begin{definition}[$w$ is an ordinary capture from $z$ of $a$ via $x$]~\\

$Ct(w,z,a,x) \triangleq Cn(x,z)\wedge \Godelnum{v_{\underline{0}}}\in z\wedge \Godelnum{v_{\underline{1}}}\in z\wedge w=\{w|\exists v(v \in a \wedge \\ \indent\mathrm{T}\mbox{{\scriptsize SUB}}(\mbox{{\scriptsize SUB}}(x,v),w)\wedge \forall y(\mathrm{T}\mbox{{\scriptsize SUB}}(\mbox{{\scriptsize SUB}}(x,v),y)\rightarrow w=y)\}$
\end{definition}

\begin{definition}[$w$ is a librationist capture from $z$ of $a$]~\\

$\mathcal{C}(w,z,a) \triangleq(\exists x)(Cn(x,z)\wedge \Godelnum{v_{\underline{0}}}\in z\wedge \Godelnum{v_{\underline{1}}}\in z\wedge w=\{w|\exists v(v \in a \wedge \mathrm{T}\mbox{{\scriptsize SUB}}(\mbox{{\scriptsize SUB}}(x,v),w)\wedge \indent\forall y(\mathrm{T}\mbox{{\scriptsize SUB}}(\mbox{{\scriptsize SUB}}(x,v),y)\rightarrow w\bleq y)  \})$
\end{definition}

\begin{definition}[Librationist specification, short hand]~\\

$\langlebar x\in a|\alpha(x)\ranglebar\triangleq\{x|\exists y(y\in a\wedge\mathrm{T}\Godelnum{\alpha(\dot{y})\wedge \dot{y}=\dot{x}}\wedge\forall z(\mathrm{T}\Godelnum{\alpha(\dot{y})\wedge \dot{y}=\dot{z}}\rightarrow x\bleq z))\}$
\end{definition}

\begin{definition}[Librationist semantic specification, short hand]~\\

$\langlebar u\in a|\mathrm{T} \mbox{\scriptsize SUB}(t,u)\ranglebar\triangleq\{x|\exists y(y\in a\wedge\mathrm{T}t\Godelnum{\frown}\Godelnum{\wedge}\Godelnum{\frown}\Godelnum{ \dot{y}=\dot{x}}\wedge\forall z(\mathrm{T}t\Godelnum{\frown}\Godelnum{\wedge}\Godelnum{\frown}\Godelnum{ \dot{y}=\dot{z}}\rightarrow x\bleq z))\}$
\end{definition}

\begin{definition}[Domination requirement with parameters from $z$ relative to $u$]~\\
	\indent$D(x,y,h,z) \triangleq\exists s,t(Cn(t,z)\wedge \Godelnum{v_{\underline{0}}}\in z\wedge \Godelnum{v_{\underline{1}}}\in z\wedge s=\mbox{\scriptsize SUB}(t,y)\wedge x=\langlebar u \in h| \mathrm{T}\mbox{\scriptsize SUB}(s,u) \ranglebar)$

\end{definition}

By manifestation point:

\begin{definition}[Domination with parameters from $z$] ~\\

$\cVvdash\forall x(x \in \mathcal{D}(h,z) \leftrightarrow \mathbf{TT}D(x,\mathcal{D}(h,z),h,z))$
\end{definition}

\begin{definition} \ \ 
	
	\bigskip
	\begin{tabular}{ll}
		
		i) & $a\overset{u}{=}b\triangleq a\in u\wedge b\in u\wedge\forall v(v\in u\rightarrow(a\in v\rightarrow b\in v))$ \\
		ii) & $\{a,b\}^{u}\triangleq\{x|x\overset{u}{=}a\vee x\overset{u}{=}b\}$\\
		iii) & $\{a\}^{u}\triangleq\{a,a\}^{u} $\\
		iv) & $\mathcal{S}(a,u)\triangleq\{x|x\in u\land (x\in a\vee x\overset{u}{=}a) \}$\\
		v) & $\emptyset\triangleq\{x|x\in u \wedge x \overset{u}{\neq} x\}$\\
		vi) & $\Omega(a,u)\triangleq\{x|x\in u\land(\forall y)(\emptyset\in y\land\forall z(z\in y\to \mathcal{S}(z,u)\in y)\to x\in y) )\}$

	\end{tabular}
	
\end{definition}
		
\begin{definition}[$y$ is a Skolem of $b$ from $u$]~\\ \label{def12}

\begin{tabular}{lll}

$S(b,y,u)$ & $\triangleq$ & $(b\in y$\\

\vspace{3pt}
\ & \ & $\wedge\forall z(z\in y\rightarrow\Omega(z,u)\in y)$\\
\vspace{2pt}
\ & \ & $\wedge\forall z(z\in y\rightarrow\mathcal{D}(z,u)\in y)$\\
\vspace{4pt}
\ & \ & $\wedge\forall w, z(w\in y \wedge \mathcal{C}(z,u,w)\rightarrow z\in y)$\\

\vspace{4pt}
\ & \ & $\wedge\forall w,z(w\in y\wedge z\in y\rightarrow\{w,z\}^{u}\in y)$\\

\ & \ & $\wedge\forall z(z\in y\rightarrow\bigcup z\in y))\bigskip$
\end{tabular}
\end{definition}




\begin{definition}[u is a Fraenkel of  x] \label{def14}

$F(x,u) \triangleq \forall y(S(\emptyset ,y,u)\rightarrow x\in y)$
\end{definition}

\begin{definition}[The $definable \ echelon$ by manifestation point]~\\ \label{def15}

\indent $\cVvdash\forall x(x\in\dot{D}\leftrightarrow\mathbf{TT}F(x,\dot{D}))$

\end{definition}

\begin{lemma}  \ \   \label{lemma1}

\bigskip$%
\begin{array}
[c]{ll}%

\vspace{3pt}

i) & \cVvdash F(\emptyset,\dot{D})\\

\vspace{4pt}

ii) & \cVvdash\forall x(F(x,\dot{D})\rightarrow F(\Omega
(x,\dot{D}),\dot{D}))\\

\vspace{4pt}
iii) & \cVvdash\forall x(F(x,\dot{D})\rightarrow F(\mathcal{D}(x,\dot{D}),\dot{D}))\\

\vspace{3pt}
iv) & \cVvdash\forall x,y(F(x,\dot{D})\wedge\mathcal{C}(y,\dot{D},w)\rightarrow F(y,\dot{D}))\\


\vspace{4pt}

v) & \cVvdash\forall x,y(F(x,\dot{D})\wedge F(y,\dot{D})\rightarrow
F(\{x,y\}^{\dot{D}},\dot{D}))\\
vi) & \cVvdash\forall x(F(x,\dot{D})\rightarrow F(\bigcup x,\dot{D}))\\
\end{array}
$\noindent

\end{lemma}
\noindent The proof of Lemma \ref{lemma1} is by invoking classical predicate logical tautologies. For the next lemma, recall Definition \ref{Kind}.

\begin{lemma} \label{lemma2}
	\ $\cVvdash KIND(\Omega(\emptyset,\dot{D}))$. 
\end{lemma}

\begin{proof}[Proof:] Adapt the proofs of Theorem 3 (i)-(iii) of \cite{Frode2012a}. \end{proof}
\begin{lemma} \ \ \label{lemma3}
	\medskip
	
	$%
	\begin{array}
	[c]{ll}%
	
	\vspace{3pt}
	i) & \cVvdash\emptyset\in\{x|F(x,\dot{D})\}\\

\vspace{3pt}

	ii) & \cVvdash\forall x(x\in\{x|F(x,\dot{D})\}\rightarrow\Omega(x,\dot{D})\in\{x|F(x,\dot{D}))\})\\
	
	\vspace{2pt}
	iii) & \cVvdash\forall x(x\in\{x|F(x,\dot{D})\}\rightarrow\mathcal{D}(x,\dot{D})\in\{x|F(x,\dot{D})\})\\
	\vspace{3pt}
	iv) & \cVvdash\forall x,y(x\in\{x|F(x,\dot{D})\}\wedge \mathcal{C}(y,\dot{D},x)\rightarrow y\in\{x|F(x,\dot{D})\})\\
	\vspace{4pt}
	v) & \cVvdash\forall x,y(x\in\{x|F(x,\dot{D})\}\wedge y\in\{x|F(x,\dot{D})\}\rightarrow\{x,y\}^{\dot{D}}\in\{x|F(x,\dot{D}))\})\\
	vi) & \cVvdash\forall x(x\in\{x|F(x,\dot{D})\}\rightarrow\bigcup x\in
	\{x|F(x,\dot{D})\})
	\end{array}
	$
\end{lemma}
\medskip

\begin{proof}[Proof:] The proof of i) is by invoking modus ascendens and
	alethic comprehension ($\mathrlap{\mathbf{C}}A_{\mathrm{M}} $) on Lemma \ref{lemma1} \hspace{1pt} i). For iii)-vi),\ invoke
	also $\pounds^{1}_{\mathrm{M}}$ ($LO1_{\mathrm{M}}$ of p. 339 of \cite{Frode2012a}), and for ii) also Lemma \ref{lemma2} is invoked, noting
	that $\cVvdash\forall x(x\in\Omega(\emptyset,\dot{D})\rightarrow\mathbb{T}x\in\Omega(\emptyset,\dot{D}))$.\end{proof}

\begin{lemma} \label{lemma4}
	$\cVvdash KIND(\{x|F(x,\dot{Z})\})$
\end{lemma}
\begin{proof}[Proof:] Adapt the proof of Theorem 3 iii) of \cite{Frode2012a} and invoke Lemma \ref{lemma3}.\end{proof}

\begin{lemma}
$\cVvdash KIND (\dot{D})$
\end{lemma}
\begin{proof}[Proof:]
A consequence of \ref{lemma4} given the definition of $\dot{D}$.
\end{proof}
\begin{definition}
	Let $\mathbb{H}$ be defined by manifestation point as in section 9 of \cite{Frode2012a}, so that we can show that $\cVvdash\forall x(x\in \mathbb{H}\leftrightarrow
	KIND(x)\wedge x\subset \mathbb{H})$.
	
	\smallskip
	\noindent Notice that we here and in the following have ``$\mathbb{H}$'' for ``$H$'' to emphasize.
\end{definition}

\begin{definition}\label{caliberdef} The caliber of sorts\\
	
	\noindent A cognomen $a$ has caliber $0\hspace{1pt}\triangleq$ $a$ is $\mathrm{T}$ or $a$ is \euro \ or for some $A$ and $x$, $a$ is $\hat \ xA$.\\
	A cognomen $a$ has caliber $n+1\triangleq$ $a$ is $\downarrow\hspace{-2pt} bc$ and the maximal caliber of $b$ and $c$ is $n$.
\end{definition}

\begin{theorem}The Regularity Rule for $\mathbb{H}$\\ \label{Regularityrule}
	
	\smallskip
	If $\cVvdash a\in \mathbb{H}$ then $\cVvdash (\exists x)(x\in a) \rightarrow (\exists x)(x\in a \wedge \forall y(y\in a \rightarrow y\notin x))$
\end{theorem}

\begin{proof}[Proof:]
	We need to extend the argument of \cite{Frode2012a}, and therefore first quote from its page 353:\\ 
	
	
	``We wrote that $H$ is a sort of iterative sorts. This holds in the following
	sense of a regularity rule:\medskip
	
	\qquad If $\cVvdash b\in H$ then $\cVvdash\exists x(x\in b)\supset
	\exists x(x\in b\wedge \forall y(y\in b\supset y\notin x))$\medskip
	
	We can justify the regularity rule briefly as follows: Suppose instead that $%
	\cVvdash b\in H$ and $\cVdash \exists x(x\in b)\wedge \forall x(x\in
	b\supset \exists y(y\in b\wedge y\in x))$. As $b$ is hereditarily kind it
	follows that $\cVvdash\exists x(x\in b)\wedge \forall x(x\in b\supset
	\exists y(y\in b\wedge y\in x))$. But the latter can only be satisfied if $b$
	is circular, a cycle or has an infinitely descending chain. Given the nature
	of $H$, it would follow that $X(0)\vDash \exists x(x\in H)$, which is
	contrary to our minimalist stipulations. Hence, $H$ only contains
	well-founded sorts as maximal members.''
	
	The proof of \citep{Frode2012a} only covers sets with caliber $0$. Suppose $n+1$ is the smallest caliber of a set $b=\downarrow\hspace{-2pt}cd$ such that $\cVvdash b\in \mathbb{H}$ and $b$ is not regular, i.e. well-founded. But this is impossible as the caliber $0$ set $\{x:x\notin c \wedge x\notin d\}$ is then also good, i.e. $\cVvdash\{x:x\notin c \wedge x\notin d\}\in \mathbb{H}$; clearly, however, $\cVvdash b\subset\{x:x\notin c \wedge x\notin d\} $, so $b$ is regular as $\{x:x\notin c \wedge x\notin d\}$ is regular.
\end{proof}

\begin{postulate}[The Skolem-Fraenkel Postulate]\label{SFP}~\\

If $b$ is good and $\mathcal{C}(w,\dot{D},b)$, $w$ is good. More precisely, $\cVvdash b\in\mathbb{H}\land\mathcal{C}(w,\dot{D},b)\to w\in\mathbb{H}$.
\end{postulate}

We write ``$SFP$'' to alert that the proof depends on the Skolem-Fraenkel Postulate.

\begin{lemma}\label{postrem1}
	Given Cannon \ref{Cannon},   Postulate \ref{SFP} holds iff $\cVvdash\dot{D}\in\mathbb{H}$.
\end{lemma}
\begin{proof}[Proof:]
Exercise.
\end{proof}
\begin{remark}
	If one  wants ``higher infinities'' in a Skolem relative sense in our framework one will need adequate postulates (in analogy with the content of Lemma \ref{postrem1}) which we intuit are then equivalent under the Cannon \ref{Cannon} with the goodness of such a Skolem relative ``higher'' infinity. As one may want a variety of postulates with a variety of strengths we do not replace the term \textit{``Skolem-Fraenkel Postulate''} with a metonym.
\end{remark}
\begin{lemma}\label{goodcapture} 
$\cVvdash \forall x,y(x \in \mathbb{H} \wedge \mathcal{C}(y,\dot{D},x)\rightarrow y \in \mathbb{H})$
\end{lemma}
\begin{proof}[Proof:]
Exercise. SFP.
\end{proof}

\begin{lemma}
$\cVvdash \forall x(x \in \mathbb{H}\rightarrow\mathcal{D}(x,\dot{D}) \in \mathbb{H})$
\end{lemma}
\begin{proof}[Proof:]
Exercise. SFP.
\end{proof}

\begin{theorem}
$\cVvdash\dot{D}\in\mathbb{H}$
\end{theorem}
\begin{proof}[Proof:]
Exercise. SFP.
\end{proof}

\begin{theorem}\label{identity} If $A(x)$ is any formula with $x $
free and only $\in$ as its non-logical symbol then $\cVvdash a\overset
{\dot{D}}{=}b\rightarrow(A^{\dot{D}}(a)\rightarrow A^{\dot{D}}(b))$. \end{theorem}

\begin{proof}[Proof:]
Exercise. SFP. \end{proof}




Let the following be an axiomatization of ZFC, where the identity sign of A7 abrreviates the consequent of A8:


\medskip
\bigskip$%
\begin{array}
[c]{ll}%
A1 & (\forall a)(\forall b)(\exists x)(a\in x\wedge b\in x)\\
A2 & (\forall a)(\exists x)(\forall y)(y\in x\leftrightarrow(\exists z)(z\in a\wedge
y\in z))\\
A3 & (\exists x)(\forall y)(y\in x\rightarrow(\exists z)(z\in x\wedge y\in z))\\
A4 & (\forall a)(\exists x)(\forall y)(y\in x\leftrightarrow(y\in a\wedge A(y)))\\
A5 & (\forall x)((\forall y)(y\in x\rightarrow A(y))\rightarrow A(x))\rightarrow(\forall
x)A(x)\\
A6 & (\forall a)((\forall x)(x\in a\rightarrow(\exists y)(A(x,y))\rightarrow(\exists
z)(\forall x)(x\in a\rightarrow(\exists y)(y\in z\wedge A(x,y)))\\
A7 & (\forall a)((\forall b)(\forall c)(b\in a\wedge c\in a\rightarrow((\exists x)(x\in b\wedge x\in c)\leftrightarrow b=c) ))\rightarrow\\
~~  & \hspace{54pt}(\exists b)(\forall c)(c\in a\rightarrow(\exists x)(\forall y)(y=x\leftrightarrow y\in c\wedge y\in b)))\\
A8 & (\forall x)(x\in a\leftrightarrow x\in b)\rightarrow(\forall u)(a\in u\leftrightarrow b\in u)\\
A9 & (\forall a)(\exists y)(\forall x)(x\in y\leftrightarrow x\subset a)
\end{array}
$

\bigskip

\begin{lemma}\label{main lemma}
If A is one of A1 to A7 then $\cVvdash A^{\dot{D}}$.
\end{lemma}
\begin{proof}[Proof:]

SFP. $A1^{\dot{D}}$ holds as $\dot{D}$ is closed under $\{a,b\}^{\dot{D}}$. $A2^{\dot{D}}$ holds as $\dot{D}$ is closed under union. $A3^{\dot{D}}$ holds as $\Omega(\overset{\dot{D}}{\emptyset},\dot{D})\in\dot{D}$. $A4^{\dot{D}}$ holds as $\dot{D}$ is closed under $\langlebar y\in a| A(y)^{\dot{D}}\ranglebar\in\dot{D}$. $A5^{D}$ holds as $\{x\in\dot{D}|\lnot A(x)^{D}\}\in\mathbb{H}$ for any first order condition $A(x)$ given the Cannon (\ref{Cannon}) and as\\ 

\indent$\cVvdash(\forall x)((\forall y)(y\in x\rightarrow y\notin\{x\in\dot{D}|\lnot A(x)^{D}\})\rightarrow x\notin\{x\in\dot{D}|\lnot A(x)^{D}\})\\
\indent\hspace{178pt}\rightarrow(\forall x)(x\notin\{x\in\dot{D}|\lnot A(x)^{D}\}$))\\ 

\noindent given the Regularity Rule (Theorem \ref{Regularityrule}). $A6^{D}$ holds as by theorem \ref{goodcapture} if $a\in\mathbb{H}$ and $\mathcal{C}(w,\dot{D},a)$ then $w\in\mathbb{H}$. For suppose $\cVvdash\forall x(x\in a\rightarrow(\exists y) A^{\dot{D}}(x,y))$ and let $w=\{w|\exists v(v \in a \wedge \mathrm{T}(\Godelnum{A^{\dot{D}}(\dot{v},\dot{w})})\wedge\forall y(\mathrm{T}(\Godelnum{A^{\dot{D}}(\dot{v},\dot{y})})\rightarrow w\bleq y))  \}$; it is then clear that $\cVvdash\forall x(x\in a\rightarrow(\exists y)(y\in w\wedge A^{\dot{D}}(x,y))$; so $\cVvdash(\forall a)(a\in\dot{D}\rightarrow(\forall x)(x\in a\rightarrow(\exists y)(y\in\dot{D}\wedge A^{\dot{D}}(x,y))))\rightarrow(\exists
z)(z\in\dot{D}\wedge(\forall x)(x\in a\rightarrow(\exists y)(y\in z\wedge A^{\dot{D}}(x,y)))$. $A7^{D}$ holds as for any given $a$ satisfying the antecedent we may (cfr. section \ref{Aus}) use librationist capture and $\{x|\exists y(y\in a \wedge x\in y \wedge (\forall z)(z\in y\rightarrow x\bleq z)))\}$ then serves as the required choice set. \end{proof}

\begin{lemma}[Dominationpotency]\label{dominationpower}~\\

\smallskip

\indent$\cVvdash\forall x,y(x\in\dot{D}\wedge y\in\dot{D}\rightarrow( x\subset y\leftrightarrow\exists z(\forall w(w\in z\leftrightarrow w\in x)\wedge z\in\mathcal{D}(y,\dot{D}))))$
\end{lemma}
\begin{exercise}
Prove Lemma \ref{dominationpower}. Hint: let $z=\langlebar u\in x| u\in y\ranglebar$.
\end{exercise}
\begin{corollary}[Weak power]\label{weak power}~\\

$\cVvdash\forall u(u\in\dot{D}\rightarrow\exists x(x\in\dot{D}\wedge\forall y(y\in\dot{D}\rightarrow\exists z(z\in x\wedge \forall w(w\in z\leftrightarrow w\in y\wedge w\in u)))))$
\end{corollary}
\begin{theorem}
\pounds \ interprets ZF.
\end{theorem}

\begin{proof}[Proof:]
By Corollary \ref{weak power} and Lemma \ref{main lemma}, \pounds \ interprets system S of \cite{Friedman73} which by its Theorem 1 interprets ZF.
\end{proof}

\begin{theorem}
	\pounds \ interprets ZFC.
\end{theorem}
\begin{proof}[Proof:]
Our strategy will here be to extend the interpretation  invoked in the proof of the previous theorem. Whereas \cite{Friedman73} uses the identity sign to abbreviate coextensionality, use e.g. $\bumpeq$ if needed in unravelling to interpret. Adapting Definition 3 in \cite{Friedman73} let $EQR(a)$ abbreviate\smallskip\ 

$\forall x,y(\langle x,y\rangle \in a\rightarrow(\forall z\in x)(\exists w\in y)(\langle z,w\rangle \in a)$

$\qquad\wedge(\forall w\in x)(\exists z\in y)(\langle z,w\rangle \in a))\smallskip$

$a\sim b$ abbreviate $(\exists x)(x\in\dot{D}\wedge EQR(x)\wedge\langle a,b\rangle \in x)\smallskip$

$a\in^{\ast}b$ abbreviate $\exists x(x\sim a\wedge x\in b)\smallskip$.

In the following we establish the standard Zermelian version, dubbed ``the multiplicative axiom'' by Russell and others, for $\in^{\ast}$ of the axiom of choice which is equivalent with our A7 above under the other principles of $ZF$: We assume that a set c is such that $\cvdash c\in\dot{D}\land\forall a(a \in^{\ast} c \to \exists x(x \in^{\ast} a))$ and $\cvdash c\in\dot{D}\land\forall a \forall b(a \in^{\ast} c\wedge b \in^{\ast} c \land \exists x(x \in^{\ast} a \land x \in^{\ast} b) \to (\forall y)(y\in\dot{D}\rightarrow (a\in^{\ast}y \leftrightarrow b\in^{\ast}y)).$ As $\cVvdash\dot{D}\in\mathbb{H}$, then $\cVvdash c\in\dot{D}\land\forall a(a \in^{\ast} c \to \exists x(x \in^{\ast} a))$ and $\cVvdash c\in\dot{D}\land\forall a \forall b(a \in^{\ast} c\wedge b \in^{\ast} c \land \exists x(x \in^{\ast} a \land x \in^{\ast} b) \to (\forall y)(y\in\dot{D}\rightarrow (a\in^{\ast}y \leftrightarrow b\in^{\ast}y)).$ Obviously $a\in c$ entails $a\in^{\ast}c$, and obviously $\exists x(x \in^{\ast} a)$ entails $\exists x(x \in a)$, so $\cVvdash\forall a(a \in c \to \exists x(x \in a))$. Clearly $\cVvdash a \in c\wedge b \in c \land \exists x(x \in a \land x \in b)$ entails $\cVvdash a \in^{\ast} c\wedge b \in^{\ast} c \land \exists x(x \in^{\ast} a \land x \in^{\ast} b)$, and we observe that $\cVvdash(\forall y)(y\in\dot{D}\rightarrow (a\in^{\ast}y \leftrightarrow b\in^{\ast}y))$ entails $\cVvdash(\forall y)(y\in\dot{D}\rightarrow (a\in y \leftrightarrow b\in y))$. This establishes that $\cVvdash\forall a(a \in c \to \exists x(x \in a))$ and $\cVvdash\forall a \forall b(a \in c\wedge b \in c \land \exists x(x \in a \land x \in b) \to (\forall y)(y\in\dot{D}\rightarrow (a\in y \leftrightarrow b\in y)).$ By the multiplicative axiom for $\in$, i.e. also A7, $\cVvdash\exists d(d\in\dot{D}\land\forall a(a \in c \to \exists x(x \in a \land x \in d))\land
\forall a(a \in c \to \forall x \forall y (x \in a \land x \in d \land y \in a \land y \in d \to x\overset{\dot{D}}{=} y))).$ Let $e$ be a wittness and consider first $\cVvdash\forall a(a \in c \to \exists x(x \in a \land x \in e))$. Suppose $\cVvdash b\in^{\ast}c$. Then for some $a$, $a\sim b$, $\cVvdash a\in c$ and thus $\cVvdash(\exists x)(x \in a \land x \in e))$ and thus also $\cVvdash(\exists x)(x \in^{\ast} a \land x \in^{\ast} e))$. By Lemma 14 of \cite{Friedman73} $\cVvdash a\sim b\leftrightarrow(\forall x)(x\in^{\ast}a\leftrightarrow x\in^{\ast} b)$. Thence $\cVvdash(\exists x)(x \in^{\ast} b \land x \in^{\ast} e)$. So $\cVvdash\forall a(a \in^{\ast} c \to \exists x(x \in^{\ast} a \land x \in^{\ast} e))$. Consider next $\cVvdash\forall a(a \in c \to \forall x \forall y (x \in a \land x \in e \land y \in a \land y \in e \to x\overset{\dot{D}}{=} y))).$ Suppose $\cVvdash b\in^{\ast}c$.
Then for some $a$, $\cVvdash a\sim b$ and $\cVvdash a\in c$. So $\cVvdash(\forall x) (\forall y) (x \in a \land x \in e \land y \in a \land y \in e \to x\overset{\dot{D}}{=} y)).$ Suppose $\cVvdash f\in^{\ast}a\land f\in^{\ast}e\land g\in^{\ast}a\land g\in^{\ast}e$. Then for some $h$ and $i$, $\cVvdash h\sim f\land h\in a\land h\in e$ and $\cVvdash i\sim g\land i\in a\land i\in e$. Thence $\cVvdash h\overset{\dot{D}}{=}i$. Given Theorem \ref{identity}, $\cVvdash f\sim g$. By adapting Lemma 19 of \cite{Friedman73}, $\cVvdash f\overset{\dot{D}}{=} g$. So $\cVvdash\forall u(u\in\dot{D}\to (f\in^{\ast}u\leftrightarrow g\in^{\ast}u))$. Consequently, $\cVvdash\forall a(a \in^{\ast} c \to \forall x \forall y (x \in^{\ast} a \land x \in^{\ast} e \land y \in^{\ast} a \land y \in^{\ast} e \to \forall u(u\in\dot{D}\to (f\in^{\ast}u\leftrightarrow g\in^{\ast}u)))).$ As $\cVvdash c\in \dot{D}$, $\cVvdash \{x|\exists y(y\in c\land\mathrm{T}\Godelnum{\dot{x}\in\dot{y}}\land\forall z(\mathrm{T}\Godelnum{\dot{x}\in\dot{y}}\to z\bleq x))\}\in\dot{D}$. By adapting Proposition \ref{capturechoice} $\{x|\exists y(y\in c\land\mathrm{T}\Godelnum{\dot{x}\in\dot{y}}\land\forall z(\mathrm{T}\Godelnum{\dot{x}\in\dot{y}}\to z\bleq x))\}$ can stand in for $e$ so that $\cVvdash(\exists d)(d\in\dot{D}\land\forall a(a \in^{\ast} c \to \exists x(x \in^{\ast} a \land x \in^{\ast} d))\land\forall a(a \in^{\ast} c \to \forall x \forall y (x \in^{\ast} a \land x \in^{\ast} d \land y \in^{\ast} a \land y \in^{\ast} d \to \forall u(u\in\dot{D}\to (f\in^{\ast}u\leftrightarrow g\in^{\ast}u)))))$. SFP.\end{proof}

\medskip

\section{Climbing  Mahlo Cardinals }

With manifestation points we can straightforwardly ascend relatively uncountable and inaccessible cardinals to the point that all sets are members of a Grothendieck universe and to the level of the first hyper inaccessible cardinals when presupposing further postulates as the Skolem-Fraenkel Postulate. In the following we climb much higher, and in the presentation we at points are somewhat repetitive as regards notational matters.

 We will show how we may start an ascent. Recall Definition \ref{def12}. Let $B(u,v)$ be $\exists p,q(u=<p,q>\hspace{-1pt}\wedge\hspace{1pt} q=\{z|\forall y(S(\emptyset,y,v)\wedge\forall w(w\in p\rightarrow \exists r(<w,r>\in v\wedge r\in y))\rightarrow z\in y)\}$.
$IN$ be the manifestation point of $B(u,v)$ so $\forall x,y(<x,y>\in IN \leftrightarrow \mathbf{TT}B(<x,y>,IN))$. As identity is kind-preserving, we have $\cVvdash\forall x,y(<x,y>\in IN \leftrightarrow B(<x,y>,IN))$.  Let $IN(p)$ abbreviate  $\{z|\forall y(S(\emptyset,y,IN)\wedge\forall w(w\in p\rightarrow \exists r(<w,r>\in IN\wedge r\in y))\rightarrow z\in y)\}$. 


\begin{theorem}[Transfinite Recursion on $\mathbb{H}$]

If $F:\mathbb{H}\to \mathbb{H}$ and $A$ is a kind subset of $\mathbb{H}$ then there is a function $f$ such that $\cVvdash\forall x,y(<x,y>\in f\leftrightarrow \mathbf{TT}(x\in A\wedge y=F(\{<u,v>|<u,v>\in f\wedge u\in x\}))$.
\end{theorem}
\begin{proof}[Proof:]
Take the appropriate manifestation point.
\end{proof}
\begin{corollary}
With such an $f$: $\cVvdash\hspace{2pt} <a,b>\hspace{1pt}\in\hspace{-2pt} f\Leftrightarrow  \ \cVvdash a\in A\wedge b=F(f\upharpoonright a)$.
\end{corollary}

We now define relative Mahlo-cardinals. Let $Tr^{2}(x)$ signify that $x$ is a transitive set and that all members of x are transitive and let $ORD=\{x|x\in \mathbb{H}\wedge Tr^{2}(x)\}$. Let $\mathbf{IN}=\{x|\exists y(y\in ORD\wedge x=\{z|z\in IN(y)\wedge Tr^{2}(z)\}\}$. Let $\texttt{F}(f)$ express that $f$ is a function, and let $\texttt{D}(f)$ denote the domain of $f$. $\texttt{NF}(x,f)$ stands for `f is a normal function on the ordinals in x', i.e more precisely: $$f\in\mathcal{D}(x^{2})\wedge \texttt{F}(f)\wedge \texttt{D}(f)=x\wedge(\delta\prec\eta\prec x\rightarrow f(\delta)\prec f(\eta)\prec x)\wedge (Lim(\gamma)\rightarrow f(\gamma)=\bigcup_{\xi\prec\gamma}f(\xi)$$

\begin{theorem}

 With manifestation point we define  function $C$:\\

 \noindent $\cVvdash<\alpha,\beta, \Xi>\in C\leftrightarrow\mathbf{TT}(<\alpha,\beta>\in Ord^{2}\wedge((\lnot\exists x(x\in\alpha)\wedge\lnot\exists x(x\in\beta)\rightarrow\Xi=\mathbf{IN})\wedge\\ 
\forall \kappa(\beta=\kappa +1\rightarrow  \Xi=\{\lambda|\exists\Psi(<\alpha,\kappa,\Psi>\in C\wedge\lambda\in\Psi\wedge\forall f(\texttt{NF}(\lambda,f)\rightarrow \exists\gamma(\gamma\in\Psi\wedge \gamma=f(\gamma))))\})\wedge \\  {\footnotesize \forall\kappa(\alpha=\kappa+1\rightarrow\Xi=\{\lambda|\exists\Psi(<\kappa,\beta,\Psi>\in C\wedge\lambda\in\Psi\wedge\forall\delta(\delta\prec\lambda\rightarrow \exists\Upsilon(<\kappa,\beta+\delta,\Upsilon>\in C\wedge\lambda\in\Upsilon)))\})\wedge}\\ (Lim(\alpha)\rightarrow\Xi=\{\lambda|\exists\delta\exists\Delta(\delta\prec\alpha\wedge<\delta,\beta,\Delta>\in C\wedge\lambda\in\Delta)\}) \wedge \\ (Lim(\beta)\rightarrow\Xi=\{\lambda|\exists\delta\exists\Delta(\delta\prec\beta\wedge<\alpha,\delta,\Delta>\in C\wedge\lambda\in\Delta)\})))$

\end{theorem}

\begin{corollary}~\\

 \noindent $\cVvdash<\alpha,\beta, \Xi>\in C\Leftrightarrow \ \cVvdash<\alpha,\beta>\in Ord^{2}\wedge((\lnot\exists x(x\in\alpha)\wedge\lnot\exists x(x\in\beta)\rightarrow\Xi=\mathbf{IN})\wedge\\ 
\forall \kappa(\beta=\kappa +1\rightarrow  \Xi=\{\lambda|\exists\Psi(<\alpha,\kappa,\Psi>\in C\wedge\lambda\in\Psi\wedge\forall f(\texttt{NF}(\lambda,f)\rightarrow \exists\gamma(\gamma\in\Psi\wedge \gamma=f(\gamma))))\})\wedge \\ {\small   \forall\kappa(\alpha=\kappa+1\rightarrow\Xi=\{\lambda|\exists\Psi(<\kappa,\beta,\Psi>\in C\wedge\lambda\in\Psi\wedge\forall\delta(\delta\prec\lambda\rightarrow \exists\Upsilon(<\kappa,\beta+\delta,\Upsilon>\in C\wedge\lambda\in\Upsilon)))\})\wedge} \\ (Lim(\alpha)\rightarrow\Xi=\{\lambda|\exists\delta\exists\Delta(\delta\prec\alpha\wedge<\delta,\beta,\Delta>\in C\wedge\lambda\in\Delta)\}) \wedge \\ (Lim(\beta)\rightarrow\Xi=\{\lambda|\exists\delta\exists\Delta(\delta\prec\beta\wedge<\alpha,\delta,\Delta>\in C\wedge\lambda\in\Delta)\})))$
\end{corollary}
\vspace{5pt}

We state the Mahlo-Postulate:

$$\cVvdash<\alpha,\beta>\in Ord^{2}\Rightarrow \ \cVvdash\exists x\exists\Psi(<\alpha,\beta,\Psi>\in C\wedge x\in\Psi)$$

With \pounds \ plus the Mahlo-Postulate we can define a manifestation point analogous to $\dot{D}$ above, and with appropriate postulates we achieve that the manifestation point is KIND and hereditarily KIND.

Some places in the literature it is suggested that somewhere around Mahlo is the limit for how far we can build up inaccessible cardinals from below. One may ask how far further one can press on with manifestation points such as here. Suggestions are welcome, but it seems clear to the author that we at least cannot go beyond indescribable cardinals.

\section{Conclusion}

We suggest that \pounds \ justifies useful set theoretic and mathematical principles more appropriately than alternative accounts, and the definable echelon $\dot{D}$ may serve as an attractive arena for definable mathematical analysis. We are somewhat interested in how far \pounds \ can account for higher infinities in a relative sense complying with that which was revealed by my compatriot Thoralf Skolem, but the author wants to underline that he is not convinced that venturing into the higher reaches of set theory as traditionally conceived in the end will turn out as edifying. The rationale for this doubt is the fact that \pounds \ is negjunction complete so that it is far from clear that questions concerning definable real numbers which we can ask at the quite low low and uncontroversially countable end of the definable hierarchy needs such higher reaches to be settled. It bears repetition that \pounds \ by itself does not even commit to the consistency of ZFC, and the lower reaches of the definable hierarchy needs little strength beyond \pounds. 

We speculate on whether and if so how how category theory may best be thought of librationistically. We are also concerned with a more general quest as to how we best can think of concrete crowds (properties), queues (instances of relations), crowds of queues (relations) and individuals with an extension of \pounds \ that lets ut think about them in a type free, adicity liberal and order eased manner. Such an extended librationist \textit{theory of properties} should integrate with modal logics by identifying the evaluations of the evaluation semantics set out in \cite{Frode2012b} with alternative veridicality predicates from the actual truth predicate $\mathrm{T}$, and the accessibility relations between such veridicality predicates should themselves as the veridicality predicates be members of the domain of the theory. We envision that such a librationist account of modal logics may overcome such limitations on modalities taken syntactically as pointed out by Richard Montague in \cite{Montague1963} and by others. There are, I think, other important fundamental questions that may be addressed advantageously if librationistically.

\medskip
\medskip
\medskip

\noindent
Universitetet i Oslo \ \ \ \ \ \ \ \ \ \ \ \ \ \ \ \ \ \ \ \ \ \ \ \ \ \ \ Universidade Federal do Rio Grande do Norte

\newpage 
\bibliographystyle{apalike}
\bibliography{ref}

\newpage

\end{document}